\documentclass[11pt]{article}

\usepackage{xcolor}
\usepackage{latexsym}
\usepackage{amssymb}
\usepackage{amsthm}
\usepackage{amscd}
\usepackage{amsmath}
\usepackage{mathrsfs}
\usepackage{graphicx}
\usepackage{hyperref}
\usepackage{shuffle}
\usepackage{mathtools}
\usepackage[bbgreekl]{mathbbol}
\usepackage{mathdots}

\usepackage[all]{xy}
\input xy \xyoption{frame}
\xyoption{dvips}

\usepackage[colorinlistoftodos]{todonotes}
\usepackage{fullpage}
\usetikzlibrary{chains,scopes}
\usetikzlibrary{shapes.geometric,positioning}

\DeclareSymbolFontAlphabet{\mathbb}{AMSb}
\DeclareSymbolFontAlphabet{\mathbbl}{bbold}

\theoremstyle{definition}
\newtheorem* {theorem*}{Theorem}
\theoremstyle{definition}
\newtheorem* {corollary*}{Corollary}
\newtheorem* {conjecture*}{Conjecture}
\newtheorem{theorem}{Theorem}[section]

\theoremstyle{definition}

\newtheorem* {example*}{Example}
\newtheorem* {fact}{Fact}
\newtheorem{lemma}[theorem]{Lemma}
\theoremstyle{definition}
\newtheorem{definition}[theorem]{Definition}
\theoremstyle{definition}
\newtheorem* {notation}{Notation}
\newtheorem{conjecture}[theorem]{Conjecture}
\newtheorem{proposition}[theorem]{Proposition}
\newtheorem{corollary}[theorem]{Corollary}

\newtheorem *{remark}{Remark}
\theoremstyle{definition}
\newtheorem {example}[theorem]{Example}
\theoremstyle{definition}

\theoremstyle{definition}

\theoremstyle{definition}

\xyoption{dvips}

\def\({\left(}
\def\){\right)}

\newcommand{\CC}{\mathbb{C}}
\newcommand{\QQ}{\mathbb{Q}}

\newcommand{\cQ}{\mathcal{Q}}

\newcommand{\cR}{\mathcal{R}}

\newcommand{\cS}{\mathcal{S}}

\newcommand{\cC}{\mathcal{C}}
\newcommand{\cD}{\mathcal{D}}

\def\cS{\mathcal{S}}

\def\NN{\mathbb{N}}

\def\CC{\mathbb{C}}

\def\ZZ{\mathbb{Z}}

\def\spanning{\textnormal{-span}}

\newcommand{\cL}{\mathcal{L}}

\def\barr{\begin{array}}
\def\earr{\end{array}}
\def\ba{\begin{aligned}}
\def\ea{\end{aligned}}
\def\be{\begin{equation}}
\def\ee{\end{equation}}

\def\qquand{\qquad\text{and}\qquad}
\def\quand{\quad\text{and}\quad}

\def\inv{\mathrm{inv}}

\def\I{\mathcal{I}}

\def\omdef{\overset{\mathrm{def}}}

\def\hs{\hspace{0.5mm}}

\def\ben{\begin{enumerate}}
\def\een{\end{enumerate}}

\def\hs{\hspace{0.5mm}}

\def\ellhat{\hat\ell}

\def\a{\textbf{a}}

\newcommand{\cA}{\mathcal{A}}

\renewcommand{\bar}{\overline}

\def\iR{\hat\cR}
\def\iF{\hat F}

\def\F{\mathscr{F}}

\def\arcstart{\ \xy<0cm,-.15cm>\xymatrix@R=.1cm@C=.3cm }
\newcommand{\arcstartc}[1]{\ \xy<0cm,-.15cm>\xymatrix@R=.1cm@C=#1cm}

\def\ellhat{\hat\ell}

\def\cT{\mathcal{T}}

\def\cG{\mathcal{G}}

\def\F{\mathscr{F}}
\def\G{\mathscr{G}}

\def\wB{\bar1\hs\bar  2\hs\bar 3\cdots  \bar n}

\def\Pair{\mathrm{Pair}}

\def\NCSM{\textsf{NCSP}}
\def\NCSQ{\textsf{NCSQ}}
\def\ONCSM{\NCSQ^-}
\def\ENCSM{\NCSQ^+}
\def\fl{\mathrm{fl}}

\def\NDes{\operatorname{NDes}}
\def\NNeg{\operatorname{NNeg}}

\DeclareRobustCommand*{\ora}{\overrightarrow}

\def\QQQ{\ora{\cL_n}}
\newcommand{\QQq}[1]{\ora{\cL}_{#1}}
\def\GGG{\ora{\cG_n}}

\def\sh{M}
\def\shq{\sh'}
\def\alphaq{\alpha'}

\def\Sink{\textsf{Sink}}
\def\Source{\textsf{Source}}

\def\F{F}
\def\G{G}
\def\iF{\hat F}
\def\iG{\hat G}

\def\flt{\underline}

\usepackage{enumerate}
\usepackage{centernot}
\usepackage[enableskew,vcentermath]{youngtab}
\newcommand{\Red}{\mathcal{R}}
\DeclareMathOperator{\SYT}{SYT}
\newcommand{\stair}{\delta}
\DeclareMathOperator{\sdeg}{sdeg}

\numberwithin{equation}{section} 
\allowdisplaybreaks[1]
\UseCrayolaColors 

\makeatletter
\renewcommand{\@makefnmark}{\mbox{\textsuperscript{}}}
\makeatother

\begin{document}

\title{Stanley symmetric functions for signed involutions}

\author{
Eric Marberg
\\ Department of Mathematics \\  HKUST \\ {\tt eric.marberg@gmail.com}
\and
Brendan Pawlowski \\ Department of Mathematics  \\ University of Southern California \\ {\tt br.pawlowski@gmail.com}
}

\date{}

\maketitle

\begin{abstract}
An involution in a Coxeter group has an associated set of \emph{involution words}, a variation on reduced words. These words are saturated chains in a partial order first considered by Richardson and Springer in their study of symmetric varieties. In the symmetric group, involution words can be enumerated in terms of tableaux using appropriate analogues of the symmetric functions introduced by Stanley to accomplish the same task for reduced words. We adapt this approach to the group of signed permutations.  We show that involution words for the longest element in the Coxeter group $C_n$ are in bijection with reduced words for the longest element in $A_n = S_{n+1}$, which are known to be in bijection with standard tableaux of shape $(n, n-1, \ldots, 2, 1)$.
\end{abstract}

\setcounter{tocdepth}{2}
\tableofcontents

\section{Introduction}


Let $W$ be a Coxeter group with simple generating set $S$.
A \emph{reduced word} for $w \in W$ is a minimal-length sequence $(r_1, r_2, \dots, r_{\ell})$ of simple generators $r_i \in S$ with $w = r_1r_2 \cdots r_{\ell}$.
Let $\Red(w)$ be the set of reduced words for $w$.

Of primary interest are the finite Coxeter groups of classical types A and C, given as follows.
Fix a positive integer $n$ and let $[n] = \{1,2,\dots,n\}$ and $[\pm n] = \{ \pm 1, \pm2 ,\dots, \pm n\}$.
Let $A_n = S_{n+1}$ be the group of permutations of $[n+ 1]$.
Let $C_n$ be the group of permutations $w$ of $[\pm n]$ with $w(-i) = -w(i)$ for all $i$.
Define $s_1,s_2,\dots,s_n \in A_n$ and $t_0,t_1,\dots,t_{n-1} \in C_n$ by
\be\label{st-eq}
 s_i = (i,i+1), \qquad t_0=(-1,1),  \qquand t_i = (-i-1,-i)(i,i+1)  \text{ for }i \neq 0.
\ee
Then $A_n$ is a Coxeter group 
relative to the generating set
$S = \{s_1,s_2,\dots,s_n\}$
while $C_n$ is a Coxeter group 
relative to the generating set 
$S = \{t_0,t_1,\dots,t_{n-1}\}$.
We refer to elements of $C_n$ as \emph{signed permutations}.

Each finite Coxeter group contains a unique element of maximal length,
where the \emph{length} of an element $w$ refers to the common length of any word in $\cR(w)$.
Let $w_n^A$ and $w_n^C$ denote the longest elements of $A_n$ and $C_n$.
Then $w_n^A$ is the permutation given in one-line notation by $(n+1)n\cdots 321$
while
$w_n^C$ is the signed permutation given by the negation map $i \mapsto -i$.
There are attractive product formulas for the number of reduced words for both of these permutations:
\begin{equation}
\label{eq:product-formulas}
\left|\Red\(w_n^A\)\right| = \frac{{n+1 \choose 2}!}{\prod_{i=1}^n (2i-1)^{n-i+1}} \qquand    \left|\Red\(w_n^C\)\right| = \frac{(n^2)!}{n^n \prod_{i=1}^{n-1} [i(2n-i)]^i}.
\end{equation}
Stanley proved the first of these identities \cite[Corollary 4.3]{Stanley} and conjectured the second, which was later shown by Haiman \cite[Theorem 5.12]{HaimanReduced}.

Let $\SYT(\lambda)$ be the set of \emph{standard Young tableaux} of shape $\lambda$.
Define \[\delta_n=(n,n-1,\dots,2,1)\] and write $(n^n)$ for the partition with $n$ parts of size $n$.
The identities \eqref{eq:product-formulas} are equivalent to
$
\left|\Red\(w_n^A\)\right| = |\SYT(\delta_n)| $ and $ \left|\Red\(w_n^C\)\right| = |\SYT((n^n))|
$
via  the \emph{hook-length formula} \cite[Corollary 7.21.6]{EC2}.
As one would expect from this formulation, there are natural bijective proofs 
of the identities \eqref{eq:product-formulas}, due to Edelman and Greene \cite{EG} in type A
and to Haiman \cite{HaimanReduced} and Kra\'skiewicz \cite{kraskiewicz-insertion} in type C.
 
The main result of this paper is a product formula similar to \eqref{eq:product-formulas} for the cardinality of a set of reduced-word-like objects
associated to $w_n^C$.
%
%
Write $\ell : W \to \NN$ for the length function of the Coxeter system $(W,S)$.
There exists a unique associative product $\circ : W\times W \to W$
with $s\circ s = s$ for any $s \in S$ and $u\circ v = uv$ for any $u,v \in W$ such that
$\ell(uv) = \ell(u) + \ell(v)$ \cite[Theorem 7.1]{Humphreys}.
This is sometimes called the \emph{Demazure product} or \emph{Hecke product} of $(W,S)$.
The pair $(W,\circ)$ is sometimes called the \emph{0-Hecke monoid} of $(W,S)$.

Let $\I(W) = \{ y \in W: y=y^{-1}\}$ be the set of involutions in $W$.
This set is preserved by the conjugation action $ w : y \mapsto w^{-1} \circ y \circ w$ 
of the 0-Hecke monoid $(W,\circ)$.
Indeed, it is a straightforward exercise from the exchange principle
for Coxeter systems (see \cite[\S1.5]{CCG})
to check the identity
\be
\label{sys-eq}
s \circ y \circ s =
\begin{cases}
sys & \text{if $\ell(ys) > \ell(y)$ and $ys \neq sy$} \\
ys & \text{if $\ell(ys) > \ell(y)$ and $ys = sy$} \\
y & \text{if $\ell(ys) < \ell(y)$} \end{cases}
\qquad\text{for $y \in \I(W)$ and $s \in S$},
\ee
which is equivalent to \cite[Lemma 3.4]{H2}.
An \emph{involution word} for $y \in \I(W)$ is a minimal-length sequence
$(r_1, r_2, \ldots, r_l)$ of simple generators $r_i \in S$ such that
\begin{equation*}
    y = r_l \circ (\cdots \circ (r_2 \circ (r_1 \circ 1 \circ r_1) \circ r_2) \circ \cdots) \circ r_{l}.
\end{equation*}
The parentheses make clear how to evaluate the right hand expression using \eqref{sys-eq},
but are actually superfluous since $\circ$ is associative.
Let $\iR(y)$ be the set of involution words for $y \in \I(W)$. 
This set is always nonempty,
with $\iR(1) = \{\emptyset\}$ where $\emptyset$ is the empty word.
Define $\ellhat(y)$ for $ y \in \I(W)$ to be the common length of any word in $\iR(y)$.

\begin{example} \label{ex:inv-word}
Let $s_i  \in A_n= S_{n+1}$ and $t_i \in C_n$ be as in \eqref{st-eq}. In $A_2$, we have
\[
s_2 \circ (s_1 \circ 1 \circ s_1) \circ s_2 = s_2 \circ s_1 \circ s_2 = s_2 s_1 s_2 
\quand
s_1 \circ (s_2 \circ 1 \circ s_2) \circ s_1 = s_1 \circ s_2 \circ s_1 = s_1s_2 s_1
\]
and it holds that $w^A_2 = s_1s_2s_1 = s_2s_1s_2$ and $\iR(w_2^A) = \{ (s_1,s_2), (s_2,s_1)\}$.
In $C_2$, we have
\[t_0 \circ (t_1 \circ (t_0 \circ 1 \circ t_0) \circ t_1) \circ t_0
        = t_0 \circ (t_1 \circ t_0 \circ t_1) \circ t_0
        = t_0 \circ t_1t_0t_1 \circ t_0
        = t_0t_1t_0t_1 = t_1t_0t_1t_0
        =w^C_2\]
and  
$t_1 \circ (t_0 \circ  (t_1 \circ 1 \circ t_1) \circ t_0) \circ t_1 = w^C_2$
and  it holds that
$\iR(w^C_2) = \{(t_0, t_1, t_0), (t_1, t_0, t_1)\}$.
\end{example}

Involution words first appeared in work of Richardson and Springer \cite{RichSpring,RichSpring2}, and have since been studied by various authors: Can, Joyce and Wyser \cite{CJW}, the authors and Hamaker \cite{HMP1, HMP2, HMP3, HMP4, HMP5}, Hu and Zhang \cite{HuZhang1, HuZhang2}, Hultman \cite{H2}, and Hansson and Hultman \cite{HH}. In \cite{HMP1}, the authors and Hamaker showed that
\begin{equation} \label{eq:type-A-inv-enum}
|\iR(w_n^A)| = {{p+1\choose 2} + {q+1 \choose 2} \choose {p+1 \choose 2}}|\SYT(\stair_{p})| |\SYT(\stair_{q})|
\end{equation}
where $p = \lfloor \frac{n}{2} \rfloor$ and $q = \lceil \frac{n}{2} \rceil$, 
and conjectured the following theorem, which   is our main result.

\begin{theorem}
\label{thm:main-enum} 
  For any positive integer $n$, it holds that 
  $|\iR(w_n^C)| = |\SYT(\stair_n)| = |\Red(w_n^A)|$.
\end{theorem}
 
 
    There is an algebraic approach to enumerating  $\Red(w_n^A)$, $\Red(w_n^C)$, $\iR(w_n^A)$, and $\iR(w_n^C)$ by means of certain generating functions called \emph{Stanley symmetric functions}.
    We write $[x_1 x_2 \cdots]f$ for the coefficient of a square-free monomial in 
    a homogeneous symmetric function $f$.
  The  Stanley symmetric functions of interest, which will be defined 
  in Section~\ref{prelim-stan-sect},
  have the following properties:
\begin{itemize}
    \item The (type A) Stanley symmetric function $\F_w$ of $w \in A_n$ has $[x_1 x_2 \cdots] \F_w = |\Red(w)|$.
    \item The (type C) Stanley symmetric function $\G_w$ of $w \in C_n$ has $[x_1 x_2 \cdots] \G_w = 2^{\ell(w)} |\Red(w)|$.
    \item The (type A) involution Stanley symmetric function $\iF_y$ of $y \in \I(A_n)$ 
    is a multiplicity-free sum of certain instances of $\F_w$, and 
    has $[x_1 x_2 \cdots] \iF_y = |\iR(y)|$.
    \item The (type C) involution Stanley symmetric function $\iG_y$ of $y \in \I(C_n)$ 
    is a multiplicity-free sum of certain instances of $\G_w$, and 
    has $[x_1 x_2 \cdots] \iG_y = 2^{\hat \ell(y)}|\iR(y)|$.
\end{itemize}
There are expressions for $\F_{w^A_n}$, $\G_{w^C_n}$, and $\iF_{w^A_n}$ as \emph{Schur functions} $s_\lambda$,  \emph{Schur $Q$-functions} $Q_\lambda$, and \emph{Schur $S$-functions} $S_\lambda$.
For the definitions of these symmetric functions, see Section~\ref{sym-sect}.
 The identities \eqref{eq:product-formulas} and \eqref{eq:type-A-inv-enum}
are corollaries of these formulas:

\begin{theorem}[{Stanley \cite[Corollary 4.2]{Stanley}}]
\label{intro-stan-thm}
It holds that
$\F_{w_n^A} = s_{\stair_n}$.
\end{theorem}

\begin{theorem}[{Worley \cite[Eq.\ (7.19)]{Worley}; Billey and Haiman \cite[Proposition 3.14]{BH}}]
\label{intro-wor-thm}
It holds that
\[\G_{w_n^C} =Q_{(2n-1,2n-3,\dots,3,1)} = S_{(n^n)}.\]
\end{theorem}

\begin{theorem}[{Hamaker, Marberg, and Pawlowski \cite[Corollary 1.14]{HMP4}}]
It holds that
\[\iF_{w_n^A} =2^{-q}Q_{(n,n-2,n-4,\dots)} = s_{\delta_{p}} s_{\delta_{q}}\]
where $p = \lfloor \frac{n}{2} \rfloor$ and $q = \lceil \frac{n}{2} \rceil$.
\end{theorem}

We prove Theorem~\ref{thm:main-enum} enumerating $\iR(w_n^C)$ by adding an entry for $\iG_{w_n^C}$ to this list.
Here, we define $\G_w$ for (unsigned) permutations $w$ by identifying $w\in A_n$ with the signed permutation in $C_{n+1}$
mapping $i \mapsto w(i)$ and $-i \mapsto -w(i)$ for $i \in [n+1]$.

\begin{theorem} \label{thm:main-stanley} 
It holds that
$\iG_{w_n^C} = \G_{w_n^A} = S_{\stair_n}$.
\end{theorem}

Our proof in Section~\ref{graph-sect} of this result proceeds as follows. One can define $\iG_{w_n^C}$ as a sum $\sum_{v \in \cA_n} \G_v$ indexed by a certain set $\cA_n$ of signed permutations $v \in C_n$, the \emph{atoms} of $w_n^C$. The transition formula of Lascoux-Sch\"utzenberger \cite{lascouxschutzenbergertree} as adapted by Billey \cite{BilleyTransitions} generates various identities between sums of type C Stanley symmetric functions. Work of Lam implies that $\G_{w_n^A} = S_{\stair_n}$ \cite{TKLam}, and we apply Billey's transition formula iteratively to rewrite $\G_{w_n^A}$ as the sum 
$\sum_{v \in \cA_n} \G_v$. 
The fact that this is possible is somewhat miraculous.
Our arguments rely heavily on a recent  characterization of the atoms of $w_n^C$
by the first author and Hamaker \cite{HMP6}.

\subsection*{Acknowledgements}

This work was partially supported by HKUST grant IGN16SC11.
We are grateful to Zach Hamaker for helpful feedback and discussions.

\section{Preliminaries}

\subsection{Symmetric functions}\label{sym-sect}

Fix a partition $\lambda = (\lambda_1 \geq \lambda_2 \geq \cdots \geq \lambda_k > 0)$.
The \emph{Young diagram} of $\lambda$ is the set of pairs $D_\lambda \omdef = \{ (i,j ) : i \in [k]\text{ and }j \in [\lambda_i]\}$, which we envision as a collection of left-justified boxes oriented as in a matrix.
%
A \emph{semistandard tableau} of shape $\lambda$ is a filling of the boxes of the Young diagram
$D_\lambda$ by positive integers, such that each row is weakly increasing from left to right
and each column is (strictly) increasing from top to bottom.
Such a tableau is \emph{standard} if its boxes contain exactly the numbers $1,2,\dots,|\lambda|$.

Similarly, a \emph{marked semistandard tableau} of shape $\lambda$ is a filling of the Young diagram of $\lambda$ by numbers from the alphabet of primed and unprimed positive integers
 $\{1,2,3,\dots\} \sqcup \{ 1',2',3',\dots\}$ such that
(i) the rows and columns are weakly increasing under the order $1' < 1 < 2' < 2 < \cdots$,
(ii) no unprimed letter $i$ appears twice in the same column,   
and (iii) no primed letter $i'$ appears twice in the same row.

Assume $\lambda $ is a strict partition, i.e., has all distinct parts.
A \emph{marked semistandard shifted tableau} of shape
$\lambda$
is a filling of the \emph{shifted Young diagram} $  \{ (i,i+j-1) : (i,j) \in D_\lambda\}$ with primed and unprimed positive integers
satisfying properties (i)-(iii) from the previous paragraph.
A semistandard marked (shifted) tableau $T$ of shape $\lambda$ is \emph{standard} 
if exactly one of $i$ or $i'$ appears in $T$ for each $i =1,2,\dots, |\lambda|$.

Given a (marked) semistandard (shifted) tableau $T$,
  write
$x^T$ for the monomial formed by replacing the boxes in $T$
containing $i$ or $i'$ by $x_i$ and then
multiplying the resulting variables.

\newcommand{\oneprime}{1'}
\newcommand{\fourprime}{4'}
\newcommand{\twoprime}{2'}
\begin{example}
If $T$, $U$, and $V$ are the tableaux of shape $\lambda = (4,3,1)$ given by
\[
    T = {\small \young(2223,334,5)}
    \qquand
    U = {\small \young(\oneprime113,\oneprime3\fourprime,5)}
    \qquand
    V= {\small \young(1\twoprime33,:\twoprime46,::5)}
\]
then $T$ is semistandard, $U$ is marked and semistandard,
and $V$ is marked, semistandard, and shifted.
We have $x^T = x_2^3 x_3^3 x_4x_5$ and $x^U = x_1^4 x_3^2 x_4 x_5$ and
$x^V = x_1x_2^2 x_3^2x_4x_5x_6$.
\end{example}

\begin{definition}
Let $\lambda$ be a partition and let $\mu$ be a strict partition.
The \emph{Schur function} of $\lambda$, the \emph{Schur $S$-function} 
of $\lambda$, and the \emph{Schur $Q$-function} of $\mu$ are then the respective sums
\[s_{\lambda} \omdef = \sum_{T} x^T, \qquad S_\lambda \omdef = \sum_U x^U,
\qquand 
Q_{\mu} \omdef = \sum_{V} x^V
\]
where $T$ runs over all semistandard tableaux of shape $\lambda$,
$U$ runs over all semistandard marked tableaux of shape $\lambda$,
and $V$ runs over all marked semistandard shifted tableaux of shape $\mu$.
\end{definition}

\def\iprime{i'}
\def\jprime{j'}
\def\kprime{k'}

The power series $s_\lambda$, $S_\lambda$, and $Q_\mu$ are all symmetric functions.
For example, we have
\[ 
\ba
4s_{(2,1)}
 = Q_{(2,1)} &=  
\sum_{i<j<k} \(
 x^{\tiny{\young(ij,:k)}} +  x^{\tiny{\young(\iprime j,:k)}} +
  x^{\tiny{\young(i\jprime,:k)}} +  x^{\tiny{\young(\iprime \jprime,:k)}} +
   x^{\tiny{\young(ij,:\kprime)}} +  x^{\tiny{\young(\iprime j,:\kprime)}} +
    x^{\tiny{\young(i\jprime,:\kprime)}} +  x^{\tiny{\young(\iprime \jprime,:\kprime)}}
\)
\\&\quad +
\sum_{i<j} \(
x^{\tiny{\young(ii,:j)}} + x^{\tiny{\young(\iprime i,:j)}} +
x^{\tiny{\young(ii,:\jprime)}} + x^{\tiny{\young(\iprime i,:\jprime)}}
+
x^{\tiny{\young(i\jprime,:j)}} + x^{\tiny{\young(\iprime\jprime,:j)}} +
x^{\tiny{\young(i\jprime,:\jprime)}} + x^{\tiny{\young(\iprime\jprime,:\jprime)}}
\) 
\\
&
=\sum_{i<j<k} 8x_ix_j x_k  + \sum_{i <j }4x_i^2 x_j  +  \sum_{i <j }4x_i x_j^2
\ea
\]
and
$
S_{(2,1)} = 
Q_{(2,1)} + 
Q_{(3)}
$.

The Schur functions $s_\lambda$, with $\lambda$ ranging over all partitions,
form a basis for the algebra $\Lambda$ of symmetric functions.
Similarly, the Schur $Q$-functions $Q_\mu$, with $\mu$ ranging over all strict partitions,
form a basis for the subalgebra $\Gamma\subset \Lambda$ generated by the odd-indexed power sum symmetric 
functions. Each Schur $Q$-function is itself Schur-positive, i.e., a linear combination of Schur functions
with positive integer coefficients.

The set of Schur $S$-functions, with $\lambda$ ranging over all partitions,
is not linearly independent, but also spans the subalgebra $\Gamma$. The set 
$\{ S_\lambda : \lambda\text{ is a strict partition}\}$ is a second basis for $\Gamma$.
For more properties of these functions, see \cite[Chaper I, \S3]{Macdonald} (for $s_\lambda$),
\cite[Chapter III, \S8]{Macdonald} (for $Q_\lambda$), and
\cite[Chapter III, \S8, Ex.\ 7]{Macdonald} (for $S_\lambda$).

\subsection{Stanley symmetric functions}\label{prelim-stan-sect}

We review the definitions of the Stanley symmetric functions (see \cite{BilleyTransitions,BH,FominKirillov,Stanley}) and involution Stanley symmetric functions (see \cite{HMP1,HMP4}) mentioned in the introduction.

\begin{definition}
The \emph{type A Stanley symmetric function} associated to $w \in A_n = S_{n+1}$ is  
\[ \F_w \omdef = \sum_{\a \in \cR(w)} \sum_{(i_1 \leq i_2 \leq \cdots \leq i_l) \in \cC(\a)} x_{i_1}x_{i_2}\cdots x_{i_l}\]
where for a reduced word $\a = (s_{a_1}, s_{a_2},\cdots, s_{a_l})$, the set $\cC(\a)$ consists of all weakly
increasing sequences of positive integers $i_1\leq i_2 \leq \dots \leq i_l$
such that if $a_j > a_{j+1}$ then $i_j < i_{j+1}$.
\end{definition}

Each $\F_w$ is a linear combination of Schur functions
with positive integer coefficients \cite{EG}.
For example,  
$ \F_{w^A_2} = \sum_{i\leq j<k} x_ix_jx_k + \sum_{i<j\leq k} x_ix_jx_k  = s_{(2,1)}$
as
$\cR(w^A_2) = \{ (s_1,s_2,s_1),(s_2,s_1,s_2)\}$.

\begin{definition}
The \emph{type C Stanley symmetric function} associated to $w \in C_n$ is
\[ \G_w \omdef = \sum_{\a \in \cR(w)} \sum_{(i_1 \leq i_2 \leq \cdots \leq i_l) \in \cD(\a)} 2^{| \{ i_1,i_2,\dots,i_l\}|} x_{i_1}x_{i_2}\cdots x_{i_l}\]
where for a reduced word $\a = (t_{a_1},t_{a_2},\cdots ,t_{a_l})$, the set $\cD(\a)$ consists of all weakly
increasing sequences of positive integers $i_1\leq i_2 \leq \dots \leq i_l$
such that if $a_{j-1} < a_j > a_{j+1}$ for some $j \in [l-1]$ then either $i_{j-1} < i_j \leq i_{j+1}$
or $i_{j-1} \leq i_j < i_{j+1}$.
\end{definition}

Each $\G_w$ is a
linear combination of Schur $Q$-functions with positive integer coefficients \cite[Theorem 3.12]{TKLam}.
It is an instructive exercise to check that $\G_{w^C_2} = Q_{(3,1)} = S_{(2,2)}$ as predicted
by Theorem~\ref{intro-wor-thm};
the details are more involved than in our calculation of $\F_{w^A_2}$, however.

\begin{remark}
The finite Coxeter groups of classical type B 
are the same as the groups $C_n$,
but there is a distinct notion of type B Stanley symmetric functions.
These only differ from $\G_w$
by a scalar factor, however:
 the type B Stanley symmetric function of $w \in C_n$ is $2^{-\ell_0(w)} \G_w$
where $\ell_0(w)$ is the number of indices $i \in [n]$ with $w(i) < 0$; see \cite{BilleyTransitions,BH,FominKirillov}.
\end{remark}

\begin{notation}
The symbols for Stanley symmetric functions
are somewhat inconsistent across the literature.
The use of $F_w$ for type A Stanley symmetric functions, 
following \cite{Stanley},
 is fairly widespread.
 Nevertheless, these functions are denoted $G_w$ in  \cite{BH,BL},
 while  in \cite{BilleyTransitions,BL} the type C Stanley symmetric functions are denoted $F_w$.
Some authors have also used $F_w$ \cite{BH} and $G_w$ \cite{BilleyTransitions,TKLamThesis,TKLam}
for the type B Stanley symmetric functions mentioned in the previous remark.
\end{notation}

There is a unique injective group homomorphism $\iota : A_{n-1} \hookrightarrow C_n$ with
$\iota(s_i) = t_i$ for $i \in [n-1]$. 
If $w \in A_{n-1} = S_n$
then $\iota(w)$
is the signed permutation with $\pm i \mapsto \pm w(i)$ for each $i \in [n]$.
We define 
\[
\G_w \omdef = \G_{\iota(w)}\qquad\text{for $w \in A_{n-1}$.}
\]
Although $s_i \mapsto t_i$ induces a bijection $\cR(w)\to\cR(\iota(w))$,
it is not obvious from the definitions how to relate $\F_w$ and $\G_w$ for $w \in A_{n-1}$.
There is a simple connection, however.
Define 
\[\Lambda \omdef = \QQ\spanning\{ s_\lambda\}\qquand \Gamma \omdef = \QQ\spanning\{ Q_\mu\} = \QQ\spanning\{ S_\mu\}\]
where the first span is over all partitions $\lambda$ and the second two are over all strict partitions $\mu$.
The \emph{superfication map} $\phi : \Lambda \to \Gamma$ is the linear map with $\phi(s_\lambda) = S_\lambda$ for all partitions $\lambda$. This is well-defined since each $S_\lambda$ is a linear combination of $S_\mu$'s with $\mu$ strict.

\begin{theorem}[{Lam \cite[Theorem 3.10]{TKLamThesis}}]
\label{super-thm}
If $w \in A_{n-1}$ then $\phi(\F_w) = \G_w$.
\end{theorem}

We turn to involution Stanley symmetric functions.
Let $(W,S)$ be a Coxeter system with length function $\ell$. Recall the definition of 
the Demazure product $\circ : W\times W \to W$ from the introduction.


\begin{definition}
For each $y \in \I(W) = \{ z \in W : z=z^{-1}\}$ let 
 $\cA(y)$ be the set of elements $w \in W$ with minimal length such that $w^{-1}\circ w = y$.
 The elements of this set are the \emph{atoms} of $y$.
\end{definition}

The associativity of $\circ$ implies that the set of involution words $\iR(y)$ for $y \in \I(W)$
is the disjoint union $\bigsqcup_{w \in \cA(y)} \cR(w)$.
The involution length of $y$ is  $\ellhat(y) = \ell(w)$ for any $w \in \cA(y)$.

\begin{definition}
The \emph{type A} and \emph{type C involution Stanley symmetric functions} 
associated to $y \in \I(A_n)$ and $z \in \I(C_n)$ are 
$ \iF_y \omdef = \sum_{w \in \cA(y)} \F_w $ and $ \iG_z \omdef = \sum_{w \in \cA(z)} \G_w$, respectively.
\end{definition}

Since $\F_w$ is Schur-positive and $\G_w$ is Schur-$Q$-positive,
it holds by construction that $\iF_y$ and $\iG_z$ are respectively Schur-positive and Schur-$Q$-positive.
For $\iF_y$, a stronger statement holds: if $\kappa(y)$ is the number of 2-cycles in $y \in \I(A_n)$,
then $2^{\kappa(y)} \iF_y$ is also Schur-$Q$-positive \cite[Corollary 4.62]{HMP4}.
We do not know if $\iG_z$ has any stronger positivity property along these lines; see Section~\ref{positivity-sect}.

\begin{example}
From Example~\ref{ex:inv-word}, we see that 
$\cA(w^C_2) = \{ t_0t_1t_0, t_1t_0t_1\}$.
Therefore
\[
\iG_{w^C_2} 
=
 \sum_{\substack{i \leq j \leq k \\ i<j\text{ or }j<k}} 2^{|\{i,j,k\}|} x_ix_jx_k + \sum_{i \leq j \leq k} 2^{|\{i,j,k\}|} x_ix_jx_k = Q_{(2,1)} + Q_{(3)}
= S_{(2,1)}.
\]
\end{example}

Define $\iG_y \omdef = \iG_{\iota(y)}$ for $y \in \I(A_{n-1})$.
Since $\cA(\iota(y)) = \iota(\cA(y))$, the following holds:

\begin{corollary}
If $y \in \I(A_{n-1})$ then $\iG_y = \phi(\iF_y)$.
\end{corollary}


\subsection{Transition formulas}

We use the term \emph{word} to refer to a finite sequence of nonzero integers.
The \emph{one-line representation} of a signed permutation $w \in C_n$ is the word $w_1w_2\cdots w_n$ where we set $w_i = w(i)$.
We usually write $\bar{m}$  in place of $-m$ so that, for example,  
the eight elements of $C_2$ 
are
$ 12,$ $\bar 1 2,$  $1 \bar 2$, $\bar 1\hs \bar 2$, $2 1$,  $\bar 2 1$, $ 2 \bar 1$, and $\bar 2\hs \bar1$.
In this notation, the longest element of $C_n$ is
\[ w^C_n = \wB.\]
The map  $w_1w_2\cdots w_n \mapsto w_1w_2\cdots w_n (n{+}1)$
is an inclusion $C_n \hookrightarrow C_{n+1}$.
We do not distinguish between $w$ and its image under this map.
If $w \in C_n$ then the words $w_1w_2\cdots w_n$ and $w_1w_2\cdots w_n(n+1)(n+2)\cdots(n+m)$ 
represent the same signed permutation for all $m \in \NN$.

Let $w \in C_n$.
Define $\inv_\pm(w)$ as the number of pairs $(i,j) \in [\pm n]\times [\pm n]$
with $i<j$ and $w_i > w_j$.
Define $\ell_0(w)$ as the number of integers $i \in [n]$ with $w_i < 0$.

\begin{lemma}[{See \cite[\S3]{BilleyTransitions}}] 
The length function of $C_n$ has the formula $\ell(w)= \frac{1}{2}\( \inv_\pm(w) + \ell_0(w)\)$.
\end{lemma}

A \emph{reflection} in a Coxeter group is an element conjugate to a simple generator.
With our notation as in \cite[\S3]{BilleyTransitions}, the reflections in $C_n$ are the following elements:
\ben
\item[(1)] $s_{ii} \omdef= 1\cdots \bar i \cdots n = (i, \bar i)$ for $i \in [n]$.
\item[(2)] $s_{ij} = s_{ji} \omdef= 1 \cdots \bar j \cdots \bar i \cdots n = (i,\bar j)(\bar i,j)$ for $i,j \in [n]$ with $i<j$.
\item[(3)] $t_{ij} = t_{ji} \omdef= 1\cdots j \cdots i \cdots n = (i,j)(\bar i,\bar j)$ for $i, j\in [n]$ with $i<j$.
\een
Observe that $t_0=s_{11}$ and $t_i = t_{i,i+1}$ 
and $s_{ij} = s_{ii} t_{ij} s_{ii} = s_{jj} t_{ij} s_{jj}$ for $i,j \in [n]$ with $i<j$.
 If $u,v \in C_n$ are any elements and $t\in C_n$ is a reflection such that $v =ut$ and $\ell(v) =\ell(u)+1$,
then we write $u \lessdot v$, so that $\lessdot$ is the covering relation of 
the \emph{Bruhat order} of $C_n$.

\begin{lemma}[{\cite[Lemmas 1 and 2]{BilleyTransitions}}]
\label{gendes-lem}
Let $w= w_1w_2\cdots w_n \in C_n$ and $i,j \in [n]$.
\ben
\item[(a)] One has $w\lessdot ws_{ii}$ if and only if 
$ w_i>0 $ and $ -w_i<e<w_i$ $\Rightarrow$ $e \notin \{w_1,w_2,\dots,w_{i-1}\}$.

\item[(b)] If $i<j$ and $w_i >0$, then $w \lessdot ws_{ij}$ if and only if
\[ 
0 < -w_j < w_i
\qquand
\begin{cases}
-w_j < e < w_i \ \Rightarrow\ e \notin \{ w_1,w_2,\dots,w_{i-1}\},
\\
-w_i < e < w_j \ \Rightarrow\ e \notin \{ w_{i+1},w_{i+2},\dots,w_{j-1}\}.
\end{cases}
\]

\item[(c)] If $i<j$ then $w\lessdot wt_{ij}$ if and only if 
\[ w_i < w_j \qquand w_i<e<w_j \ \Rightarrow\ e \notin \{w_{i+1},w_{i+2},\dots w_{j-1}\}.\]

\een
\end{lemma}

For example, it holds that
 $1\bar 2 43 \lessdot 1\bar 2 43 \cdot t_{14} = 3\bar 24 1$ while $1243 \centernot \lessdot 1243\cdot t_{14} = 3241$, and it holds that $3   2 4  \bar 1 \lessdot 3   2 4  \bar 1 \cdot s_{14} = 1  2 4  \bar 3$ while $3  \bar 2 4  \bar 1 \centernot\lessdot 3  \bar 2 4  \bar 1 \cdot s_{14} = 1  \bar 2 4  \bar 3$.

   Lemma~\ref{gendes-lem}(c) says that $w \lessdot wt_{ij}$ if and only if $w_i < w_j$ and no entry 
   in $w_1w_2\cdots w_n$ between positions $i$ and $j$ 
is between $w_i$ and $w_j$ in value. 
Lemma~\ref{gendes-lem}(a) can be described in the same way using 
``symmetric'' one-line notation:
one has $w \lessdot ws_{ii}$ if and only if $\bar{w_{ i} }< w_i$ and no number between $\bar{w_{i}}$ and $w_i$  appears in the word  $\bar{w_{i-1}} \cdots \bar{w_2}\hs \bar{w_1} w_1w_2\cdots w_{i-1}$.
One can express Lemma~\ref{gendes-lem}(b) similarly. 
We frequently only need the following special cases of these conditions:

\begin{lemma}\label{des-lem}
Let $w =w_1w_2\cdots w_n\in  C_n$ and $j \in \{2,3,\dots,n\}$. Assume $w_1 > 0$.
\ben
\item[(a)] $wt_0 \lessdot w$.

\item[(b)] $ws_{1j} \lessdot w$ if and only if $w_1 < -w_j$ and no $e \in \{ w_{2},w_{3},\dots,w_{j-1}\}$
has
$w_j<e<-w_1$.

\item[(c)] $wt_{1j} \lessdot w$ if and only if 
$
w_1 > w_j
$
and no $e \in \{ w_{2},w_{3},\dots,w_{j-1}\}$
has
$w_1 > e > w_j$.

\een
\end{lemma}

Let $[m,n] = \{ i \in \ZZ : m \leq i \leq n\}$.
For $w \in C_n$ and $j \in [n]$, we define three sets:
\be\label{stt-eq}
\ba
    \cT_j^+(w) &\omdef = \{wt_{jk} : k \in [j+1,n+1], w \lessdot wt_{jk}\} \subseteq C_{n+1},\\
    \cT_j^-(w) &\omdef = \{wt_{ij} : i \in [j-1], w \lessdot wt_{ij}\} \subseteq C_n,\\
    \cS_j(w) &\omdef = \{ws_{ij} : i \in [n], w \lessdot ws_{ij}\} \subseteq C_n.
\ea
\ee
The next theorem, which is analogous to the transition formulas of Lascoux and Sch\"utzenberger \cite{lascouxschutzenbergertree}, is the main technical tool we require to work with type C Stanley symmetric functions.

\begin{theorem}[{Billey \cite[Lemma 8]{BilleyTransitions}}] \label{thm:type-C-transitions}
  If $w \in C_n$ and $j \in [n]$ then 
  \[ \sum_{v \in \cT^+_j(w)} \G_{v} =  \sum_{v \in \cS_j(w)} \G_v +  \sum_{v \in \cT^-_j(w)} \G_{v}. \]
  \end{theorem}

This result leads to an effective algorithm for computing the Schur $Q$-expansion of $\G_w$.

\begin{theorem}[{Billey \cite[Corollary 9]{BilleyTransitions}}]
 \label{thm:transition-tree} Suppose $w \in C_n$.
  \begin{enumerate}[(a)]
    \item If $w_1 <\dots < w_r < 0 < w_{r+1} < \dots < w_n$ for some $r \in [n]$, then
$\G_w = Q_{(-w_1, -w_2, \ldots, -w_r)}.$

    \item Suppose $(r,s) \in [n] \times [n]$ is lexicographically maximal such that $r < s$ and $w_r > w_s$. Let $v = wt_{rs}$. Then
$ \G_{w} =  \sum_{{i \in [n], v \lessdot vs_{ir}}} \G_{vs_{ir}} + \sum_{{i \in [r-1], v\lessdot vt_{ir}}} \G_{vt_{ir}} $. 

  \end{enumerate}
\end{theorem}

The theorem gives a recursion for $\G_w$ which terminates when $w$
is strictly increasing. Billey shows that this recursion always terminates in a finite number of steps \cite[Theorem 4]{BilleyTransitions}.


\begin{example}
The results of \cite{HMP6} (see Section~\ref{atoms-sect}) imply that
$
        \cA(\bar 8 \hs \bar 7 \hs \bar 6 \hs \bar 5 \hs \bar 4 \hs \bar 3 \hs \bar 2 \hs \bar 1) =
        \{\bar 8 \hs \bar 6 \hs \bar 4 \hs \bar 2  1  3  5  7\}
$
so
\[
\iG_{\bar 8 \hs \bar 7 \hs \bar 6 \hs \bar 5 \hs \bar 4 \hs \bar 3 \hs \bar 2 \hs \bar 1} = \G_{\bar 8 \hs \bar 6 \hs \bar 4 \hs \bar 2  1  3  5  7}= Q_{(8,6,4,2)}
\] by Theorem~\ref{thm:transition-tree}(a). 
It follows that the number of involution words for $\bar 8 \hs \bar 7 \hs \bar 6 \hs \bar 5 \hs \bar 4 \hs \bar 3 \hs \bar 2 \hs \bar 1$  times $2^{8+6+4+2}$
is equal to the number of marked standard shifted tableaux of shape $(8,6,4,2)$.
This example generalizes in a straightforward way from $C_8$ to any $C_n$.
\end{example}

\begin{example}
In the graphs $\ora\cG$ below, the identity $\sum_{\{u \to w\} \in \ora\cG} \G_u = \sum_{\{w \to u\} \in \ora\cG} \G_u$
is an instance of Theorem~\ref{thm:type-C-transitions} for each boxed vertex $w$, with $j$ as the index of the underlined letter of $w$.
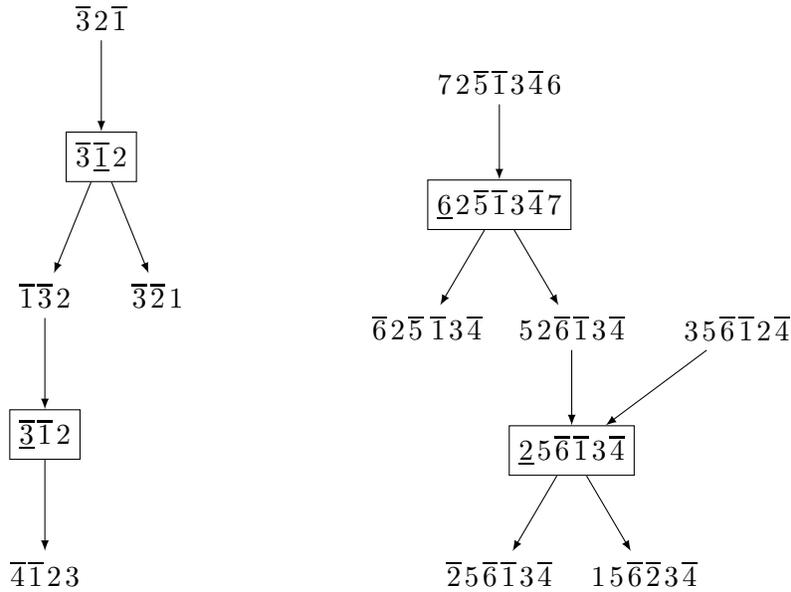
    \begin{center}
        \begin{tikzpicture}[level distance=3em, edge from parent/.style={draw, -latex},yscale=1.6]
            \node {$\bar 3 \hs 2 \hs \bar 1$}
                child { node[rectangle,draw] {$\bar 3 \hs \underline{\bar 1} \hs 2$}
                    child { node {$\bar 1 \hs \bar 3 \hs 2$}
                        child { node[rectangle, draw] {$\underline{\bar 3} \hs \bar 1 \hs 2$}
                            child { node {$\bar 4 \hs \bar 1 \hs 2 \hs 3$} }
                        }
                     }
                    child { node {$\bar 3 \hs \bar 2 \hs 1$} }
                 };
        \end{tikzpicture} \hspace{2cm}
        \begin{tikzpicture}[level distance=3em, yscale=1.6]
            \node (11) {$7 \hs 2 \hs\bar 5 \hs \bar 1 \hs 3 \hs \bar 4 \hs 6$};
            \node[rectangle,draw] (21) [below=of 11] {$\underline{6} \hs2 \hs\bar 5 \hs \bar 1 \hs 3 \hs \bar 4 \hs7$} edge[latex-] (11);
            \node (31) [below=of 21, xshift=-2.5em] {$\bar 6 \hs 2 \hs \bar 5 \hs \hs \bar 1 \hs 3 \hs \bar 4$} edge [latex-] (21);
            \node (32) [below=of 21, xshift=2.5em] {$5 \hs2 \hs\bar 6 \hs \bar 1 \hs 3 \hs \bar 4$} edge [latex-] (21);
            \node[rectangle,draw] (42) [below=of 32] {$\underline{2} \hs5 \hs\bar 6 \hs \bar 1 \hs 3 \hs \bar 4$} edge [latex-] (32);
            \node (33) [right=of 32, xshift=-1.25em] {$3 \hs 5 \hs \bar 6 \hs \bar 1 \hs 2 \hs \bar 4$} edge [-latex] (42);
            \node (52) [below=of 42, xshift=-2.5em] {$\bar 2 \hs 5 \hs \bar 6 \hs \bar 1 \hs 3 \hs \bar 4$} edge [latex-] (42);
            \node (53) [below=of 42, xshift=2.5em] {$1 \hs 5 \hs\bar 6 \hs \bar 2 \hs 3 \hs \bar 4$} edge [latex-] (42);
        \end{tikzpicture}
    \end{center}
   For example,  the graph on the left is constructed as follows: from each unboxed vertex $w$ which is not increasing (starting with $\bar 3 2 \bar 1$), draw an arrow to the boxed vertex $wt_{jk}$ and underline the letter of $wt_{jk}$ in position $j$;
   then for each boxed vertex $w$, draw an arrow to any unboxed vertices of the form $wt_{ij} \in \cT^-_j(w)$ or $ws_{ij} \in \cS_j(w)$.
    In reading these graphs one should keep in mind that we identify $w_1 w_2 \cdots w_n \in C_n$ with $w_1 w_2 \cdots w_n(n+1) \in C_{n+1}$. For both graphs $\ora\cG$ it holds that 
     \[
     \sum_{u \in \Sink(\ora\cG)} \G_u = \sum_{u \in \Source(\ora\cG)} \G_u
     \]
      and it follows, using Theorem~\ref{thm:transition-tree}(a) and Lemma~\ref{lem:transition-graphs} below, that
    $\G_{\bar 3 \hs2 \hs \bar 1} = \G_{\bar 3 \hs \bar 2 \hs 1} + \G_{\bar 4 \hs \bar 1 \hs 2 \hs 3} = Q_{(3,2)} + Q_{(4,1)}$
    and
    $\G_{7 \hs2 \hs\bar 5 \hs \bar 1 \hs 3 \hs \bar 4 \hs 6} + \G_{3 \hs5 \hs\bar 6 \hs \bar 1 \hs 2 \hs \bar 4} = \G_{\bar 6 \hs 2 \hs \bar 5 \hs \bar 1 \hs 3 \hs \bar 4} + \G_{\bar 2 \hs 5 \hs \bar 6 \hs \bar 1 \hs 3 \hs \bar 4} + \G_{1 \hs 5 \hs \bar 6 \hs \bar 2 \hs 3 \hs \bar 4}.$
    \end{example}

Suppose $\ora\cG$ is a directed graph with $v$ a vertex. Write $\sdeg(v)$ for the indegree of $v$ minus its outdegree, and $\deg(v)$ for the indegree of $v$ plus its outdegree.
Let $\Source(\ora\cG)$ be the set of vertices in $\ora\cG$ with indegree zero and 
let $\Sink(\ora\cG)$
be the set of vertices with outdegree zero.
We say that a vertex $v$ is an \emph{interior vertex} if it is neither a source nor a sink.

\def\cV{\mathcal{V}}

\begin{lemma} \label{lem:transition-graphs} Let $\ora\cG$ be a finite bipartite directed graph with vertex set $\cV$ and bipartition $\cV = \cV^- \sqcup \cV^+$. Assume that $\sdeg(u) =0$ for all interior vertices $u \in \cV^+$. Suppose $f : \cV \to A$ is a function to an abelian group $A$ such that if $v \in \cV^-$ then 
\[f(v) = \sum_{\{v \to w\} \in \ora\cG} f(w) -  \sum_{\{u \to v\} \in \ora\cG} f(u) = 0.\] Then
$\sum_{u \in \Sink(\ora\cG)} \deg(u) f(u) = \sum_{u \in \Source(\ora\cG)} \deg(u) f(u).$
\end{lemma}

\begin{proof} By assumption $\sum_{\{x \to y\} \in \ora\cG} (f(y)-f(x)) = \sum_{v \in \cV^-}\( \sum_{v \to w} f(w) - \sum_{u \to v} f(u)  \) = 0$, while
$\sum_{\{x \to y\} \in \ora\cG} (f(y)-f(x)) = \sum_{u \in \cV^+} \sdeg(u)f(u) = \sum_{u \in \Sink(\ora\cG)} \deg(u) f(u) - \sum_{u \in \Source(\ora\cG)} \deg(u) f(u)$
since $\sdeg(u) =0$ for all interior vertices $u \in \cV^+$.
\end{proof}

\section{Atoms}\label{atoms-sect}

The atoms of the longest  element $w^C_n \in \I(C_n)$ have a number of special properties,
which we review in this section.
Let $S\subset \ZZ$ be a set of integers.
A \emph{perfect matching} on a set $S$ is a set $M$ of pairwise disjoint 2-element subsets $\{i,j\}$, referred to as \emph{blocks},
whose  union is $S$. A perfect matching $M$ is \emph{symmetric} if $\{i,j \} \in M$ implies $-\{i,j\}  = \{-i,-j\} \in M$,
and \emph{noncrossing} if it does not occur that $i<a<j<b$ for any $\{i,j\},\{a,b\} \in M$. 
Let $\NCSM(n)$ denote the set of noncrossing, symmetric, perfect matchings on the set $[\pm n]$.
The three elements of $\NCSM(3)$ are
\[\{ \{\pm 1\}, \{\pm 2\}, \{ \pm 3\}\}, 
\qquad
 \{\pm \{1,2\},\{\pm 3\}\}, 
 \qquand \{\{\pm 1\},\pm \{2,3\}\}.\]
In general, $|\NCSM(n)| = \binom{n}{\lfloor n/2\rfloor}$; see \cite[A001405]{OEIS}.
We emphasize the following basic observation:

\begin{fact}
If $M \in \NCSM(n)$ and $\{i,j\} \in M$, then 
$i$ and $j$ have the same sign or $i=-j$.
\end{fact}

If  $w=w_1w_2\cdots w_n$ is a word
then we write $[[w]]$ for the subword 
formed by omitting each repeated letter after its first appearance.
For example, $[[31231124]] = 3124$.
Suppose $M$ is a symmetric, noncrossing, perfect matching on a subset of $[\pm n]$.
Define
\[
\Pair(M) \omdef = \{ (a,-b): \{a,b\} \in M \text{ and }0<a<b\} \sqcup \{ (-a,-a) : \{-a,a\} \in M \text{ and }0<a\}.
\]
Let $(a_1,b_1),(a_2,b_2),\dots,(a_l,b_l)$
(respectively, $ (c_1,d_1), (c_2,d_2),\dots,(c_l,d_l)$) be the elements $\Pair(M)$ 
listed in order such that $b_1<b_2<\dots<b_l$
(respectively, $ c_1<c_2<\dots<c_l$).
Define the words
\[ \alpha_{\min}(M) \omdef = [[a_1b_1a_2b_2\cdots a_lb_l]]
\qquand
\alpha_{\max}(M) \omdef = [[c_1d_1c_2d_2\cdots c_ld_l]].
\]
If $M \in \NCSM(n)$, then $\alpha_{\min}(M)$ and $\alpha_{\max}(M)$
contain exactly one letter from $\{\pm i\}$ for each $i \in [n]$,
so  may be interpreted as elements of $C_n$.
\begin{example} \label{ex:min-max-atoms}
If
$
M = \{ \pm\{1,3\},\pm\{4,7\}, \pm\{5,6\}, \{\pm 8\}\}
$
 then
\[ \Pair(M) =\{ (-8,-8), (4,-7), (5,-6), (1,-3)\} =  \{ (-8,-8), (1,-3), (4,-7), (5,-6)\} \}\]
and
$
\alpha_{\min}(M) = \bar8 4 \bar75\bar 6 1 \bar 3
$
and
$
\alpha_{\max}(M) = \bar 8 1\bar 3 4\bar75\bar 6.
$
Note that these words do not represent signed permutations since no block of $M$ contains $\pm 2$.
\end{example}

If $u$ and $v$ are words, both with $n\geq i+2$ letters, 
then we write $u \vartriangleleft_i v$ to mean that 
\be
\label{tri-eq}
u_iu_{i+1}u_{i+2} = cab,\quad v_iv_{i+1}v_{i+2} = bca,\quand u_j=v_j\text{ for }j \notin \{i,i+1,i+2\}
\ee
for some numbers $a<b<c$.
Define $<_\cA$ as the transitive closure of the relations $\vartriangleleft_i$ for all $i \geq 1$.
Equivalently, $<_\cA$ is the transitive closure of the relation on words with
\be\label{<A-eq} \cdots cab\cdots <_\cA \cdots bca \cdots\ee
whenever $a<b<c$ and the corresponding ellipses mask identical subsequences.
This relation is a partial order since it 
is a sub-relation of lexicographic order. We apply $<_\cA$ to signed permutations via their one-line representations.
Define 
\[\cA_n \omdef= \cA(\wB)\subset C_n\]
 and for each $M \in \NCSM(n)$ let 
$\cA_M \omdef = \{ w \in C_n : \alpha_{\min}(M) \leq_\cA w \leq_\cA \alpha_{\max}(M)\}. $

\begin{example}
    Let $M = \{\{\pm 1\}, \pm\{2,3\}, \pm\{4,5\}\} \in \NCSM(5)$. The interval $\cA_M$ is
    \begin{center}
        \begin{tikzpicture}[node distance=0.5cm]
            \node (1) {$\bar 1  2  \bar 3  4 \bar 5$};
            \node [below=of 1, xshift=-1cm] (2a) {$\bar 1  2  4  \bar 5 \hs \bar 3$} edge  (1);
            \node [below=of 1, xshift=1cm] (2b) {$2\hs \bar 3 \hs \bar 1 \hs 4 \hs \bar 5$} edge (1);
            \node [below=of 2a] (3a) {$\bar 1  4  \bar 5 2 \bar 3$} edge   (2a);
            \node [below=of 2b] (3b) {$2  \bar 3  4  \bar 5  \bar 1$} edge  (2b);
            \node [below=of 3a] (4a) {$4  \bar 5\hs  \bar 1  2  \bar 3$} edge  (3a);
            \node [below=of 3b] (4b) {$2  4  \bar 5\hs  \bar 3\hs  \bar 1$} edge  (3b);
            \node [below=of 4a, xshift=1cm] (5) {$4  \bar 5  2  \bar 3\hs  \bar 1$} edge  (4a) edge  (4b);
        \end{tikzpicture}
    \end{center}
\end{example}

\begin{theorem}[See \cite{HMP6}]\label{interval-thm}
Relative to $<_\cA$, there is a disjoint poset decomposition 
\[\cA_n = \bigsqcup_{M \in \NCSM(n)} \cA_M.\]
In particular, each $x \in \cA_n$ belongs to $\cA_M$ for a unique $M \in \NCSM(n)$,
and if $x\in \cA_M$ and $y \in \cA_N$ for $M,N \in \NCSM(n)$, then $x\leq_\cA y$ only if $M=N$.
\end{theorem}

This result remains true when $n=0$ if we take both $\cA_0$ and $C_0$ to be
 the singleton set containing just the empty word $\emptyset$,
 and define $\alpha_{\min}(M) = \alpha_{\max}(M) = \emptyset$
 if $M =\varnothing \in \NCSM(0)$.

\begin{proof}
 \cite[Theorem 5.6]{HMP6} describes the connected components of $\cA(z)^{-1}  = \{w^{-1} : w \in \cA(z)\}$ under $\leq_\cA$ 
 for any $z \in \I(C_n)$ in terms of symmetric noncrossing matchings, and implies that $\cA_n^{-1} = \bigsqcup_{M \in \NCSM(n)} \cA_M$. The theorem follows since $\cA_n =\cA_n^{-1}$ by
  \cite[Proposition 2.7]{HMP6}.
\end{proof}

Fix an atom $w \in  \cA_n$.
We define the \emph{shape} of $w$ to be the unique matching $\sh(w) \in \NCSM(n)$
with $w \in \cA_{\sh(w)}$.
It is helpful to understand how the shape of $w$ can be extracted from
the one-line representation $w_1w_2\cdots w_n$.
This can be done as follows.

From $w \in \cA_n$, we produce a sequence of words $w^0,w^1,\dots,w^l$.
Start by letting $w^0 = w_1w_2\cdots w_n$.
For each $i>0$, form $w^i$ by removing an arbitrary descent from $w^{i-1}$,
where a \emph{descent} in a word $a_1a_2\cdots a_n$ is a consecutive subword $a_ia_{i+1}$ with $a_i > a_{i+1}$.
The sequence terminates when we obtain an increasing word $w^l$. 
Let $\{c_1,c_2,\dots,c_k\}$ be the set of letters in $w^l$ and 
suppose $q_ip_i$ is the descent removed from $w^{i-1}$ to form $w^i$.
We then define
\[
\NNeg(w) \omdef= \{ -c_1,-c_2,\dots,-c_k\}
\qquand
\NDes(w) \omdef= \{ (q_1,p_1), (q_2,p_2),\dots,(q_l,p_l)\}.
\] 
We refer to $\NNeg(w)$ and $\NDes(w)$ as the \emph{nested negated set} and \emph{nested descent set} of $w$.

\begin{theorem}[See \cite{HMP6}]
\label{ndes-thm}
No matter how the words $w^0,w^1,\dots,w^l$ are constructed,
we have:
\ben
\item[(a)] $\NNeg(w) = \{ a : \{-a,a\} \in \sh(w)\text{ and }0<a\}$.
\item[(b)] $\NDes(w) = \{ (a,-b) : \{a,b\} \in \sh(w)\text{ and }0<a<b\}.$
\een
\end{theorem}

\begin{proof}
This is equivalent to \cite[Theorem-Definition 3.10]{HMP6} since
$\cA_n = \cA_n^{-1}$ \cite[Proposition 2.7]{HMP6}.
\end{proof}

\begin{example}
    If $M = \{ \pm\{1,3\},\{\pm 2\}, \pm\{4,7\}, \pm\{5,6\}, \{\pm 8\}\}$ and $w = \bar 8 \hs\bar2 1  4  \bar 7  5  \bar 6 \hs \bar 3$ then
    \begin{equation*}
        w = \bar 8\hs\bar2  1  4  \bar 7  5  \bar 6 \hs \bar 3 \vartriangleleft_6 \bar 8\hs\bar2  1  4  \bar 7 \hs \bar 3  5  \bar 6 \vartriangleleft_4 \bar 8\hs\bar2  1  \bar 3  4  \bar 7  5  \bar 6 = \alpha_{\max}(M),
    \end{equation*}
    so $w \in \cA_M$. There are two ways to progressively remove descents from $w$ as described above:
    \begin{align*}
        w = w^0 = \bar 8\hs\bar2  1  4  \bar 7  5  \bar 6 \hs \bar 3, \quad 
        w^1 = \bar 8\hs\bar2  1  5  \bar 6 \hs \bar 3, 
        \quad w^2 = \bar 8\hs\bar2  1  \bar 3, \quad w^3 = \bar 8\hs\bar2\\
        w = w^0 = \bar 8\hs\bar2  1 4  \bar 7  5  \bar 6 \hs \bar 3, \quad 
        w^1 = \bar 8\hs\bar2  1  4  \bar 7 \hs \bar 3, \quad 
        w^2 = \bar 8\hs\bar2  1  \bar 3, \quad w^3 = \bar 8\hs\bar2
    \end{align*}
    Both give $\NNeg(w) = \{2,8\}$ and $\NDes(w) = \{(1,-3), (4,-7), (5,-6)\}$ as claimed by Theorem~\ref{ndes-thm}.
\end{example}

The preceding theorem has several implications, starting with the following observation.

\begin{corollary}
\label{descent-cor}
Let $M \in \NCSM(n)$ and $w \in \cA_M$.
If $w_i>w_{i+1}$ for some $i \in [n-1]$ then $0 < w_i < -w_{i+1}$ and $\{ w_i,-w_{i+1}\} \in M$.
The word $w_1w_2\cdots w_n$ therefore
contains no consecutive subwords of the form $ba$ where $0<a<b$ or $a<b<0$,
or of the form $cba$ where $a<b<c$.
\end{corollary}

Subwords in the following lemma need not be consecutive.

\begin{lemma}\label{order-lem}
Let $M \in \NCSM(n)$ and $w \in \cA_M$.
Suppose $0<a<b<c<d$.
\ben
\item[(a)] If $\{a,d\},\{b,c\} \in M$ then $a\bar d b\bar c$ is a subword of $w_1w_2\cdots w_n$.
\item[(b)] If $\{a,b\},\{\pm c\} \in M$ then $\bar c a \bar b$ is a subword of $w_1w_2\cdots w_n$.
\item[(c)] If $\{\pm a\},\{\pm b\} \in M$ then $\bar b \bar a$ is a subword of $w_1w_2\cdots w_n$.
\een
\end{lemma}

\begin{proof}
Suppose $S =-S\subset [\pm n]$ is a union of blocks in $M$.
Given $v \in \cA_n$, let $v_S$  be the subword of $v_1v_2\cdots v_n$
with all letters not in $S$ removed.
It follows from Corollary~\ref{descent-cor} that if $u,v\in \cA_n$ and $u\leq_\cA v$ then $u_S \leq_\cA v_S$.
Let $u = \alpha_{\min}(M)$ and $v = \alpha_{\max}(M)$.
If $\{a,d\},\{b,c\} \in M$ and $S = \{ \pm a, \pm b, \pm c, \pm d\}$
then
it follows that
$ u_S  \leq_\cA w_S \leq_\cA v_S $.
Since in this case $u_S = v_S = a \bar d b \bar c$, we deduce that
$w_S=a\bar d b\bar c$ is a subword of $w_1w_2\cdots w_n$. Parts (b) and (c) follow similarly.
\end{proof}

Let $w \in \cA_n$.
If $1 \leq i < j \leq n$ are indices such that
$0 < w_i < -w_j$ and $\{w_i,-w_j\} \in \sh(w)$, then we say that $i$ and $j$ are \emph{complementary indices} in $w$.
If $i \in [n]$ is such that $w_i < 0$ and $\{ \pm w_i\} \in \sh(w)$,
then we say that $i$ is a \emph{symmetric index} in $w$.

\begin{corollary}\label{compl-cor}
If $w \in \cA_n$
then each $i \in [n]$ is symmetric or part of a complementary pair in $w$.
\end{corollary}

Therefore, if $w \in \cA_n$ and $i \in [n]$ is such that $w_i>0$, then there exists a complementary index $j \in [n]$
with $i<j$ and  $w_i < -w_j$ and $\{w_i,-w_j\} \in \sh(w)$.
In turn, since  $1$ cannot be the second index in a complementary pair, if $w \in \cA_n$ and $w_1 < 0$
then we must have $\{ \pm w_1\} \in \sh(w)$.

\begin{lemma}\label{interval-lem}
Suppose $1\leq i < j\leq n$ are complementary indices for $w \in \cA_n$ and $e \in [\pm n]$.
\ben
\item[(a)] If $w_j < e < -w_i<0$, then $e \notin \{ w_1,w_2,\dots,w_j,w_{j+1}\}$.

\item[(b)] If $0<w_i < e < -w_j$, then $e \notin \{ w_1,w_2,\dots,w_{j}\}$.
\een
\end{lemma}

\begin{proof}
We have $\{w_i,-w_j\} \in \sh(w)$, so if $e \in \{w_1,w_2,\dots,w_n\}$ and $w_i < |e| < -w_j$,
then the noncrossing matching $\sh(w)$ must contain a block $\{a,b\}$ with $|e| \in \{a,b\}$
and $w_i < a < b < -w_j$, and in this case 
$w_i w_j a \bar b$ must be a subword of $w_1w_2\dots w_n$ by Lemma~\ref{order-lem}(a).
\end{proof}

 \begin{lemma}\label{des1j-lem}
If $1\leq i < j\leq n$ are complementary indices for $w \in \cA_n$ then $ws_{ij} \lessdot w$.
\end{lemma}

\begin{proof}
This is immediate from Lemmas~\ref{gendes-lem}(b) and \ref{interval-lem}.
\end{proof}

\section{Quasi-atoms}

Given a word $w=w_1w_2\cdots w_n$
such that $|w_1|$, $|w_2|$, \dots, $|w_n|$ are distinct and nonzero,
define
$\fl_\pm(w) \in C_n$ to be the signed permutation whose one-line representation is formed by replacing each letter of $w$
by its image under the order-preserving bijection $\{ \pm w_1, \pm w_2,\dots, \pm w_n\} \to [\pm n]$.
For example, we have $\fl_\pm(3\bar 25 \bar 7) = 2 \bar 1 3\bar 4 \in C_4$.
If $M$ is a partition of a symmetric $2n$-element subset $X=-X\subset [\pm m]$,
then define $\fl_\pm(M)$ to be the partition of $[\pm n]$ formed by replacing each element of each block of $M$
by its image under the order-preserving bijection $X \to [\pm n]$.

Suppose $w \in C_n$ and $v = \fl_\pm(w_2w_3\cdots w_n) \in \cA_{n-1}$. Define $\shq(w)$
to be the unique perfect matching on $[n]\setminus\{\pm w_1\}$
with $\fl_\pm(\shq(w)) = \sh(v)$.
Since $\sh(v)$ is symmetric and noncrossing, $\shq(w)$ is symmetric and noncrossing.

The matching $\shq(w)$ may be read off directly from the one-line representation of $w$ 
by the following procedure.
Let $w^0,w^1,w^2,\dots,w^l$ be any sequence of words whose
first term is $w^0 =w_2w_3\cdots w_n$ (note the deliberate omission of $w_1$) and whose final term is strictly increasing,
in which $w^i$ for $i>0$ is formed from $w^{i-1}$ by removing a single descent $q_ip_i$.
Let $\{c_1,c_2,\dots,c_k\}$ be the set of letters in $w^l$.
Then $\shq(w)$ is the matching whose blocks consist of
$\{p,-q\}$, $\{-p,q\}$, and $\{\pm c\}$
for each descent $(q,p) \in \{ (q_1,p_1), (q_2,p_2),\dots, (q_l,p_l)\}$ and each $c \in \{ c_1,c_2,\dots,c_k\}$.
This construction is independent
of the choices of descents by Theorem~\ref{ndes-thm}.

\begin{example}
    Let $w = 3 1 6 \bar 7  4  \bar 5 \hs \bar 2$. One sequence of words $w^0, w^1, \ldots, w^l$ as described above is
    \begin{equation*}
        w^0 = 1 \hs6 \hs\bar 7 \hs 4 \hs \bar 5 \hs \bar 2, \quad w^1 = 1 \hs4 \hs\bar 5 \hs \bar 2, \quad w^2 = 1 \hs\bar 2, \quad w^3 = \emptyset,
    \end{equation*}
    so $\shq(w) = \{\pm\{1,2\}, \pm\{4,5\}, \pm\{6,7\} \}$. Setting $v = \fl_{\pm}(w_2 w_3 \cdots w_n)$, we have
    \begin{equation*}
        \alpha_{\min}(M) = 5  \bar 6  3  \bar 4  1  \bar 2 \vartriangleleft_3 5  \bar 6  1  3  \bar 4 \hs \bar 2 \vartriangleleft_1 1  5  \bar 6  3  \bar 4 \hs \bar 2 = v
    \end{equation*}
    where $M = \{\pm\{1,2\},\pm\{3,4\},\pm\{5,6\}\} = \fl_{\pm}(\shq(w))$, so $v \in \cA_M$.

\end{example}

We define $\cA_0$ to be the singleton set containing just the empty word $\emptyset$.

\begin{definition}
An element $w \in C_n$ is a \emph{quasi-atom} if
the following conditions hold:
\begin{itemize}
\item[(a)] One has $w_1>0$ and $\fl_\pm(w_2w_3\cdots w_n) \in \cA_{n-1}$, so $\shq(w)$ is defined.
\item[(b)] At most one 
block $\{a,b\} \in \shq(w)$ has $0<a<w_1<b$.

\item[(c)] No symmetric
block $\{ \pm c \} \in \shq(w)$ has $0<w_1<c$.
\end{itemize}
A quasi-atom $w$ is  \emph{odd}  if no block
$\{a,b\} \in \shq(w)$
exists with $0<a<w_1<b$;
otherwise, $w$ is \emph{even}.
We write $\cQ_n^+$ and $\cQ_n^-$ for the sets of even and odd quasi-atoms in $C_n$,
and define $\cQ_n \omdef = \cQ_n^+ \sqcup \cQ_n^-$.
\end{definition}

\begin{example} 
By convention we have 
\ben
\item[] $\cA_0 = \{\emptyset\}$,
\item[] $\cQ_1^+ = \varnothing$ and $\cQ_1^- = \{1\}$.
\een
In rank two we have:
\ben
\item[] $\cA_1 = \left\{\bar 1\right\}$,
\item[] $\cQ_2^+ = \varnothing$ and $\cQ_2^- = \left\{2 \bar 1\right\}$.
\een
In rank three we have:
\ben
\item[] $\cA_2 = \left\{\bar 2\hs \bar 1, 1  \bar 2\right\}$,
\item[] $\cQ_3^+ = \left\{ 21\bar 3\right\}$ and $\cQ_3^- = \left\{ 3   \bar 2\hs \bar 1, 3   1   \bar 2, 1 2  \bar 3\right\}$.
\een
In rank four we have:
\ben
\item[] $\cA_3 = \left\{ \bar 3 \hs \bar 2 \hs \bar 1, \bar 3  1 \bar 2, 2 \bar3 \hs\bar 1, \bar 1 2 \bar3\right\}$,
\item[] $\cQ_4^+ = \left\{3 2 \bar4\hs \bar 1, 3 \bar1  2 \bar 4\right\}$ and 
$\cQ_4^- = \left\{4 \bar 3\hs \bar 2\hs \bar 1, 4 \bar 3  1 \bar 2, 4 2 \bar 3\hs \bar 1, 4 \bar 1 2  \bar 3, 2 3 \bar 4 \bar 1, 2 \bar 1 3 \bar 4\right\}$.
\een
The sequences of cardinalities
\[
\ba
\( |\cA_n|: n=1,2,3,\dots\) &= (1,\ 2,\ 4,\ 11,\ 30,\ 101,\ 336,\ 1310,\ 5039,\ \dots) \\
\( |\cQ_n^+| : n=1,2,3,\dots\) &= (0,\ 0,\ 1,\ 2,\ 11,\ 30,\ 151,\ 501,\ 2592,\ \dots) \\
\( |\cQ_n^-|: n=1,2,3,\dots\) &= (1,\ 1,\ 3,\ 6,\ 21,\ 57,\ 228,\ 753,\ 3359,\ \dots) \\
\( |\cQ_n|: n=1,2,3,\dots\) &= (1,\ 1,\ 4,\ 8,\ 32,\ 87,\ 379,\ 1254,\ 5951,\ \dots) \\
\ea
\]
do not match any existing entries in \cite{OEIS}.
\end{example}

It can happen that $w \in \cA_n$ and $\fl_\pm(w_2w_3\cdots w_n) \in \cA_{n-1}$,
in which case $\sh(w) $ and $ \shq(w)$ are both defined but unequal. 
For quasi-atoms, however, this ambiguity does not arise:

\begin{proposition}\label{disjoint-prop}
The sets $\cA_n$ and $\cQ_n$ are disjoint.
\end{proposition}

\begin{proof}
Suppose $w \in \cA_n \cap \cQ_n$.
Since $w_1>0$, the index 1 must be complementary to some $j \in [2,n]$,
but then $0<w_1 < -w_j$ and  necessarily $\{\pm w_j\} \in \shq(w)$, contradicting the definition of $\cQ_n$.
\end{proof}

Define $<_\cQ$ as the
transitive closure of the relations $\vartriangleleft_i$ from \eqref{tri-eq} for $i\geq 2$.
This is the partial order
with $v <_\cQ w$ if and only if $v_1=w_1 $ and $v_2v_3\cdots v_n <_\cA w_2w_3\cdots w_n$,
or equivalently the transitive closure of the relation on words with
\[ x \cdots cab\cdots <_\cQ x \cdots bca \cdots\]
whenever $a,b,c,x$ are integers with 
$a<b<c$
and  the corresponding ellipses mask identical subsequences.
Each subset $\cA_M\subset \cA_n$ is a weakly connected component of the Hasse diagram of the partial order $<_\cA$. 
Since whether $w \in \cQ_n$ is even or odd depends only on the matching $\shq(w)$,
it follows that $\cQ_n^+$ and $\cQ_n^-$ 
are each unions of weakly connected components of the Hasse diagram of the partial order $<_\cQ$.

Let $e \in [n]$ and suppose $M$ is a perfect matching on $[\pm n] \setminus \{\pm e\}$
which is symmetric and noncrossing. Assume $M$ has  no blocks $\{\pm c\}$ with $0<e<c$. 
Define $\ENCSM(n,e)$ as the set of such matchings with exactly one block $\{a,b\}$ such that $0 < a<e<b$;
define $\ONCSM(n,e)$ as the set of such matchings with no blocks $\{a,b\}$ such that $0<a<e<b$.
Let 
\[
\ENCSM(n) \omdef= \bigsqcup_{e \in [n]} \ENCSM(n,e)
\qquand
\ONCSM(n) \omdef= \bigsqcup_{e \in [n]} \ONCSM(n,e).
\]
Given $M \in \NCSQ^\pm(n,e)$,
define $\alphaq_{\min}(M) \omdef= u_1u_2u_3\cdots u_n$ and $\alphaq_{\max}(M) \omdef= v_1v_2v_3\cdots v_n$ where 
\[ u_1=v_1=e\qquand u_2u_3\cdots u_n = \alpha_{\min}(M)\qquand v_2v_3\cdots v_n = \alpha_{\max}(M).\]
Finally let $\cQ_M \omdef= \{ w \in C_n : \alphaq_{\min}(M) \leq_\cQ w \leq_\cQ \alphaq_{\max}(M)\}$.

\begin{example}
We have
\[
    \{\pm\{1,7\}, \pm\{2,3\}, \pm\{5,6\}\} \in \ENCSM(7,4)
    \quand
    \{\pm\{2,7\}, \pm\{3,4\}, \pm\{5,6\}\} \in \ONCSM(7,1).
\]
If $M = \{\{\pm 3\}, \pm\{1,2\}, \pm\{4,8\}, \pm\{6,7\}\} \in \ENCSM(8,5)$, then
\[
    \alpha_{\min}(M) = 5 \hs 4 \hs \bar 8 \hs 6 \hs \bar 7 \hs \bar 3 \hs 1 \hs \bar 2 \qquad \text{and} \qquad \alpha_{\max}(M) = 5 \hs \bar 3 \hs 1 \hs \bar 2 \hs 4 \hs \bar 8 \hs 6 \hs \bar 7.
\]
\end{example}

\begin{proposition}
Relative to $<_\cQ$, there are disjoint poset decompositions
\[\cQ_n^+ = \bigsqcup_{M \in \ENCSM(n)} \cQ_M
\qquand
\cQ_n^- = \bigsqcup_{M \in \ONCSM(n)} \cQ_M.
\]
In particular, each $x \in \cQ^\pm_n$ belongs to $\cA^\pm_M$ for a unique $M \in \NCSQ^\pm(n)$,
and if $x\in \cQ_M$ and $y \in \cQ_N$ for $M,N \in \NCSQ^\pm(n)$, then $x\leq_\cQ y$ only if $M=N$.
\end{proposition}

\begin{proof}
This is clear since if $M \in \NCSQ^\pm(n)$ then $\cQ_M = \{ w \in \cQ_n : \shq(w) = M\}$.
\end{proof}

Let $w \in \cQ_n$ be a quasi-atom.
Mimicking our terminology in the previous section,
define indices $2\leq i < j \leq n$ to be \emph{complementary} in $w$ 
if $0 < w_i < -w_j$ and $\{w_i,-w_j\} \in \shq(w)$,
and define an index $2 \leq i \leq n$ to be \emph{symmetric} for $w$ if $w_i < 0$ and $\{\pm w_i\} \in \shq(w)$.
In view of Proposition~\ref{disjoint-prop}, there is no risk of these notions conflicting with our earlier definitions for atoms.

With minor changes, the technical properties of atoms in the previous section remain true for quasi-atoms.
The following summarizes the main facts we will need.
\begin{lemma}
\label{qinterval-lem}
Consider a quasi-atom $w \in \cQ_n$.
\ben
\item[(a)] Each index $i \in [2,n]$ is symmetric or part of a complementary pair for $w$.
\item[(b)] If $w_i>w_{i+1}$ for some $i \in [2,n-1]$ then $0<w_i < -w_{i+1}$.
\item[(c)] Suppose $0<a<b<c<d$.
\ben
\item[i.] If $\{a,d\},\{b,c\} \in \shq(w)$ then $a\bar d b\bar c$ is a subword of $w_2w_3\cdots w_n$.
\item[ii.] If $\{a,b\},\{\pm c\} \in \shq(w)$ then $\bar c a \bar b$ is a subword of $w_2w_3\cdots w_n$.
\item[iii.] If $\{\pm a\},\{\pm b\} \in \shq(w)$ then $\bar b \bar a$ is a subword of $w_2w_3\cdots w_n$.
\een
\item[(d)] Suppose $2\leq i < j\leq n$ are complementary indices for $w$ and $e \in [\pm n]$.
\ben
\item[i.] If $w_j < e < -w_i<0$ then $e \notin \{ w_2,w_3,\dots,w_j,w_{j+1}\}$.
\item[ii.] If  $0<w_i < e < -w_j$ then $e \notin \{ w_2,w_3,\dots,w_{j}\}$.
\een
\een
\end{lemma}

\begin{proof}
Since $\fl_\pm(w_2w_3\cdots w_n)$ is required to belong to $\cA_M$
for some matching $M \in \NCSM(n-1)$,
and since $\shq(w)$ is defined to be the matching on $[\pm n]\setminus \{\pm w_1\}$
with $\fl_\pm(\shq(w)) = M$,
these properties
just restate 
Corollary~\ref{descent-cor},
Lemma~\ref{order-lem},
Corollary~\ref{compl-cor},
and Lemma~\ref{interval-lem}.
\end{proof}

%
%

%



If $w \in \cQ_n^+$ is an even quasi-atom
then there exists a unique pair of complementary indices $2\leq i < j \leq n$ with $0<w_i < w_1 < -w_j$.
We call these the \emph{distinguished indices} of $w$.

\begin{corollary}\label{dist-lem}
Suppose $ 2\leq i<j \leq n$ are the distinguished indices of an even quasi-atom $w \in \cQ_n$.
Then  $wt_{1i} \lessdot w$ and $ws_{1j} \lessdot w$.
\end{corollary}

\begin{proof}
This is immediate from Lemmas~\ref{des-lem} and \ref{qinterval-lem}(d).
%
%
\end{proof}

\section{Transition graphs}\label{graph-sect}

We define a directed bipartite graph $\QQQ$ with vertex set $\cA_n \sqcup \cQ_n$.
We use the letter $\cL$ to denote this graph since it will later serve as one ``layer'' in a larger
graph of interest.
Each edge in $\QQQ$ will pass 
either from an
even quasi-atom to an odd quasi-atom, from an odd quasi-atom to an even quasi-atom,
or from an odd quasi-atom to an atom.
The atoms of $w^C_n=\wB$ will each have a unique incoming edge,
and all even quasi-atoms will have one incoming and one outgoing edge.
These properties will not be immediately clear from the following definition.

First suppose 
$v \in \cQ_n^+$ is an even quasi-atom.
Let $b=v_1>0$ and suppose $\{a,c\} \in \shq(v)$ is the unique block with $0<a<b<c$.
Let
 $2 \leq i < j\leq n$ be
the distinguished indices with $a=v_i$ and $c=-v_j$.
In $\QQQ$,
we define $v$ to have
 a unique incoming edge $u\to v$ 
 where
\be\label{u-eq} u =  vs_{1j} =t_{bc}v = \bar{v_j}v_2\cdots v_{j-1} \bar{v_1} v_j \cdots v_n
\ee
and a unique outgoing edge $v \to w$ where
\be\label{w-eq}
w =vt_{1i} = t_{ab} v=  v_i v_2 \cdots v_{i-1}v_1 v_{i+1} \cdots v_n.\ee
Next suppose $v \in \cA_n$. If $v_1 < 0$
then let
\be
\label{u1-eq}
u = vt_0 = \bar{v_1}   v_2\cdots v_n.
\ee
If  $v_1> 0$
and  $j \in[2,n]$ is the unique index with $\{v_1,-v_j\} \in \sh(v)$, then let
\be
\label{u2-eq}
u = vs_{1j} = t_{bc}v=   \bar{v_j} v_2\cdots v_{j-1}  \bar{v_1}   v_{j+1} \cdots v_n
\ee
where $b = v_1<-v_j=c$.
We define $v$ to have a single incoming edge $u \to v$ in $\QQQ$.
Figure~\ref{fig1} 
shows $\QQQ$ for $n=1,2,3,4$
and Figure~\ref{fig2}
shows a part of $\ora{\cL_5}$.

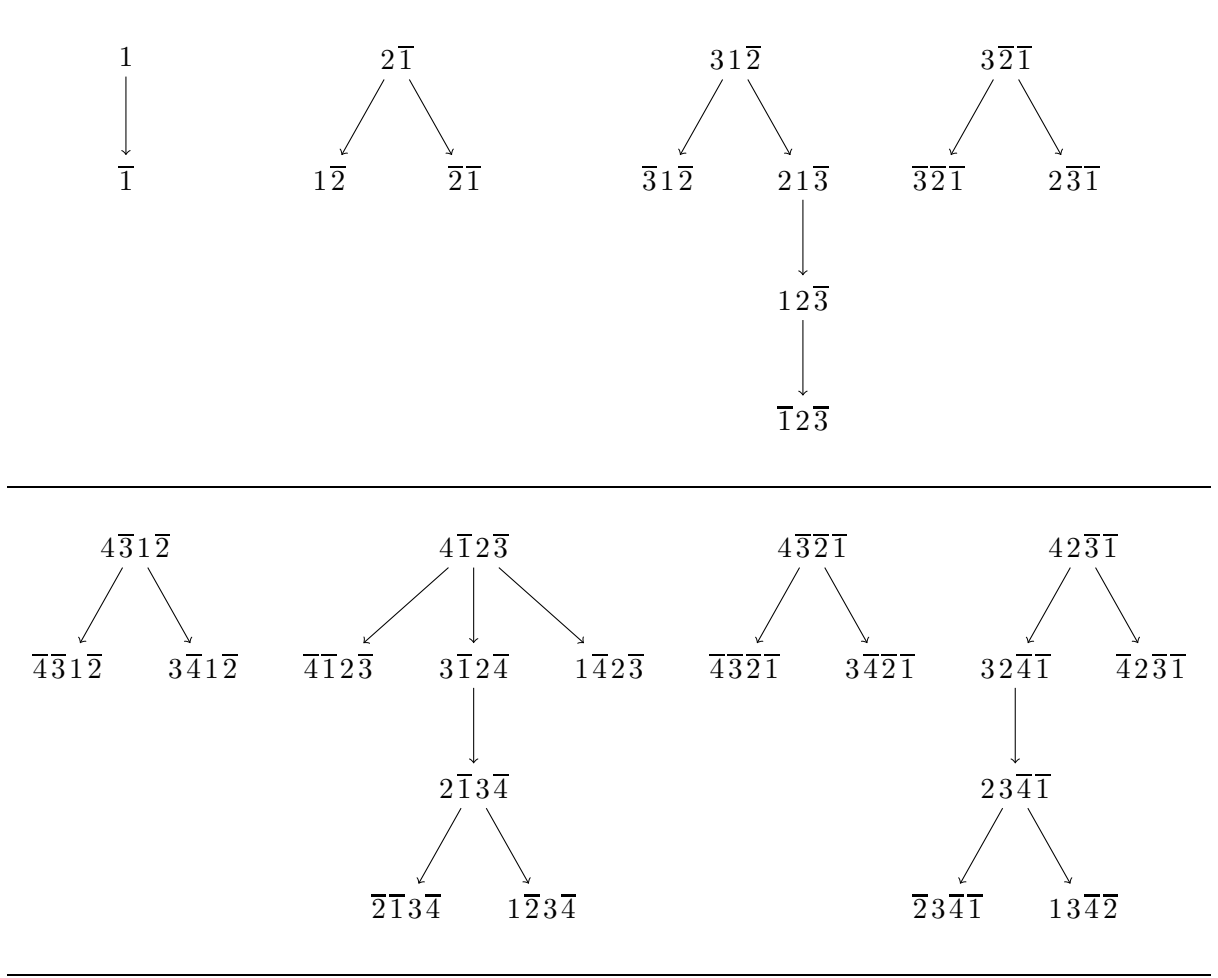
\begin{figure}[h]
\begin{center}
\begin{tabular}{c}
\hline
\\
\begin{tikzpicture}[xscale=0.9,yscale=1.6]
\node (a) at (0,0) {$1$};
\node (b) at (0,-1) {$\bar 1$};
\draw [->] (a) -- (b);
\node (2node2) at (4,0) {$2\hs \bar1$};
\node (2node1) at (3,-1) {$1\hs \bar2$};
\node (2node0) at (5,-1) {$\bar2\hs \bar1$};
\draw [->] (2node2) -- (2node0);
\draw [->] (2node2) -- (2node1);
\node (3node6) at (13,0) {$3\hs \bar2\hs \bar1$};
\node (3node0) at (12,-1) {$\bar3\hs \bar2\hs \bar1$};
\node (3node4) at (14,-1) {$2\hs \bar3\hs \bar1$};
\node (3node7) at (9,0) {$3\hs 1\hs \bar2$};
\node (3node1) at (8,-1) {$\bar3\hs 1\hs \bar2$};
\node (3node5) at (10, -1){$2\hs 1\hs \bar3$};
\node (3node3) at (10,-2) {$1\hs 2\hs \bar3$};
\node (3node2) at (10,-3) {$\bar1\hs 2\hs \bar3$};
\draw [->] (3node6) -- (3node0);
\draw [->] (3node6) -- (3node4);
\draw [->] (3node7) -- (3node1);
\draw [->] (3node7) -- (3node5);
\draw [->] (3node5) -- (3node3);
\draw [->] (3node3) -- (3node2);
\end{tikzpicture}
\\
\\\hline
\\
\begin{tikzpicture}[xscale=0.9,yscale=1.6]
\node (4a2) at (0,0) {$4\hs \bar3\hs 1\hs \bar2$};
\node (4b3) at (-1,-1) {$\bar4\hs \bar3\hs 1\hs \bar2$};
\node (4b4) at (1,-1) {$3\hs \bar4\hs 1\hs \bar2$};
\node (4a4) at (5,0) {$4\hs \bar1\hs 2\hs \bar3$};
\node (4b7) at (3,-1) {$\bar4\hs \bar1\hs 2\hs \bar3$};
\node (4b8) at (7,-1) {$1\hs \bar4\hs 2\hs \bar3$};
\node (4b9) at (5,-1) {$3\hs \bar1\hs 2\hs \bar4$};
\node (4c1) at (5,-2) {$2\hs \bar1\hs 3\hs \bar4$};
\node (4d1) at (4,-3) {$\bar2\hs \bar1\hs 3\hs \bar4$};
\node (4d2) at (6,-3) {$1\hs \bar2\hs 3\hs \bar4$};
\node (4a1) at (10,0) {$4\hs \bar3\hs \bar2\hs \bar1$};
\node (4b1) at (9,-1) {$\bar4\hs \bar3\hs \bar2\hs \bar1$};
\node (4b2) at (11,-1) {$3\hs \bar4\hs \bar2\hs \bar1$};
\node (4a3) at (14,0) {$4\hs 2\hs \bar3\hs \bar1$};
\node (4b5) at (15,-1) {$\bar4\hs 2\hs \bar3\hs \bar1$};
\node (4b6) at (13,-1) {$3\hs 2\hs \bar4\hs \bar1$};
\node (4c2) at (13,-2) {$2\hs 3\hs \bar4\hs \bar1$};
\node (4d3) at (12,-3) {$\bar2\hs 3\hs \bar4\hs \bar1$};
\node (4d4) at (14,-3) {$1\hs 3\hs \bar4\hs \bar2$};
\draw [->] (4a1) -- (4b1);
\draw [->] (4a1) -- (4b2);
\draw [->] (4a2) -- (4b3);
\draw [->] (4a2) -- (4b4);
\draw [->] (4a3) -- (4b5);
\draw [->] (4a3) -- (4b6);
\draw [->] (4a4) -- (4b7);
\draw [->] (4a4) -- (4b8);
\draw [->] (4a4) -- (4b9);
\draw [->] (4b6) -- (4c2);
\draw [->] (4c2) -- (4d3);
\draw [->] (4c2) -- (4d4);
\draw [->] (4b9) -- (4c1);
\draw [->] (4c1) -- (4d1);
\draw [->] (4c1) -- (4d2);
\end{tikzpicture}
\\
\\ \hline
\end{tabular}
\end{center}
\caption{The graphs $\QQQ$ for $n=1,2,3,4$. A vertex belongs to $\cA_n$ if it has no outgoing edges.
A vertex with an outgoing edge belongs
to $\cQ_n^+$ if it is in the top or third row, and to $\cQ_n^-$ otherwise.
}
\label{fig1}
 \end{figure}

\begin{figure}[h]
\begin{center}
\begin{tikzpicture}[xscale=1.2,yscale=1.5]
\node (a1) at (1,6) {$5\hs 3\hs \bar4\hs 1\hs \bar2$};
\node (b1) at (0,5) {$\bar5\hs 3\hs \bar4\hs 1\hs \bar2$};
\node (b2) at (2,5) {$4\hs 3\hs \bar5\hs 1\hs \bar2$};
\node (c1) at (2,4) {$3\hs 4\hs \bar5\hs 1\hs \bar2$};
\node (c2) at (6,4) {$5\hs 1\hs \bar4\hs 2\hs \bar3$};
\node (d1) at (1,3) {$\bar3\hs4\hs\bar5\hs1\hs\bar2$};
\node (d2) at (3,3) {$2\hs4\hs\bar5\hs1\hs\bar3$};
\node (d3) at (5,3) {$4\hs1\hs\bar5\hs2\hs\bar3$};
\node (d4) at (7,3) {$\bar5\hs1\hs\bar4\hs2\hs\bar3$};
\node (e1) at (4,2) {$1\hs4\hs\bar5\hs2\hs\bar3$};
\node (f1) at (4,1) {$\bar1\hs4\hs\bar5\hs2\hs\bar3$};
\draw [->] (a1) -- (b1);
\draw [->] (a1) -- (b2);
\draw [->] (b2) -- (c1);
\draw [->] (c1) -- (d1);
\draw [->] (c1) -- (d2);
\draw [->] (c2) -- (d3);
\draw [->] (c2) -- (d4);
\draw [->] (d2) -- (e1);
\draw [->] (d3) -- (e1);
\draw [->] (e1) -- (f1);
\end{tikzpicture}
\end{center}
\caption{The most interesting connected component of $\ora{\cL_5}$.
As demonstrated by this example, the graphs $\QQQ$ are not always directed forests.}
\label{fig2}
 \end{figure}
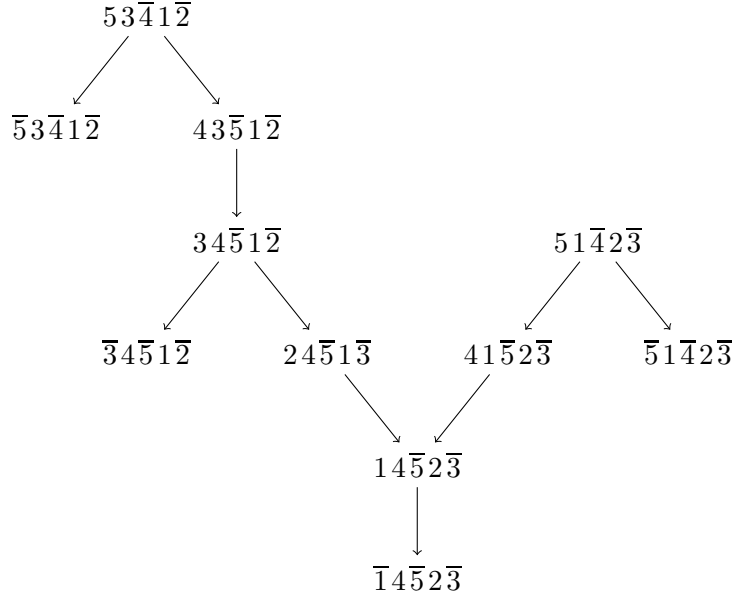

\begin{lemma}\label{prec-lem}
Let $v,v' \in \cA_n \sqcup \cQ_n^+$ and $i \in [2,n-2]$.
Define $\vartriangleleft_i $ as in \eqref{tri-eq}.
\ben
\item[(a)]  If 
 $u\to v$ is an edge in $\QQQ$ and $v\vartriangleleft_i v'$, then
 there is an edge $u' \to v'$ in $\QQQ$ with
  $u \vartriangleleft_i u'$.

\item[(b)]  If $v \to w$ is an edge in $\QQQ$ 
and $v'\vartriangleleft_i v$, then 
there is an edge $v' \to w'$ in $\QQQ$ with
$w' \vartriangleleft_i w$.
\een
\end{lemma}

\begin{proof}
First suppose $v,v' \in\cA_n$ and $v\vartriangleleft_i v'$.
Let $u\to v$ and $u'\to v '$ be the unique edges incident to $v$ and $v'$ in $\QQQ$.
If $v_1<0$ then $u = vt_0$ and $u'=v't_0$ and clearly $u\vartriangleleft_i u'$.
Assume $b=v_1 > 0$ and suppose $\{b,c\} \in \sh(v)$
is the unique block with $0<b<c$,
so that $u = t_{bc} v$ and $u' = t_{bc} v'$.
We can write $v_iv_{i+1}v_{i+2} = zxy$ where $x<y<z$.
By Corollary~\ref{descent-cor}, it must hold that $0 < z < -x$.
Since $u_2u_3\cdots u_n$ is given by replacing $\bar c$ by $ \bar b$ in $v_2v_3\cdots v_n$,
the only way we can fail to have $u \vartriangleleft_i u'$ 
is if $v_i = - c$ and $- c < v_{i+1} < - b$.
But this is impossible by Lemma~\ref{interval-lem}(a).

Now let $v,v' \in \cQ_n^+$.
Suppose $b = v_1$
and $\{a,c\} \in \shq(v)$ is the unique block with $0<a<b<c$.
Set
  $u = t_{bc}v$ and $w= t_{ab}v$
  so that 
$u\to v$ and $v \to w$ are the edges incident to $v$ in $\QQQ$.

Suppose $v\vartriangleleft_i v'$ and  
  $u' = t_{bc}v'$
so that $u'\to v'$ is an edge in $\QQQ$. Our argument is similar to the previous case.
We can write $v_iv_{i+1}v_{i+2} = zxy$ where $x<y<z$,
and $0 < z < -x$ by Lemma~\ref{qinterval-lem}(b).
Since $u_2u_3\cdots u_n$ is given by replacing $\bar c$ by $\bar b$ in $v_2v_3\cdots v_n$,
the only way we can fail to have $u \vartriangleleft_i u'$ 
is if $v_{i+1}= - c$ and $- c < v_{i+2} < - b$.
But this is impossible by Lemma~\ref{qinterval-lem}(d).

Finally suppose $v'\vartriangleleft_i v$ and
$w' = t_{ab}v'$ so that $v' \to w'$ is an edge in $\QQQ$.
We can write $v_iv_{i+1}v_{i+2} = yzx$ where $x<y<z$,
and $0 < z < -x$ by Lemma~\ref{qinterval-lem}(b).
Since $w_2w_3\cdots w_n$ is given by replacing $a$ by $ b$ in $v_2v_3\cdots v_n$,
the only way we can fail to have $w' \vartriangleleft_i w$ 
is if $v_i = a$ and $a<v_{i+1}<b$.
But this is impossible by Lemma~\ref{qinterval-lem}(d).
\end{proof}

\begin{lemma}\label{chop-lem}
Let $M \in \NCSM(n)$ and $v \in \cA_M$.
 If $v_1<0$ then define $M'$
by removing $\{\pm v_1\}$ from $M$.
Otherwise $-v_1$ and $ v_1$ belong to distinct blocks in $M$, and we  define $M'$ by removing these blocks
and then adding $\{ \pm v_1\}$.
In either case we have $M' \in \ONCSM(n)$.
If $v$  is maximal with respect to $<_\cQ$ then $\alphaq_{\max}(M') \to v$ is the unique edge incident to $v$ in $\QQQ$.
\end{lemma}

\begin{proof}
If $v_1<0$ then clearly $M' \in \ONCSM(n, v_1)$.
If $v_1>0$ and $j \in [2,n]$ is such that $0<v_1<-v_j$ and $\{v_1,-v_j\} \in M$,
then
Lemma~\ref{order-lem} implies that $M$ has no blocks $\{a,b\}$ with $a<v_1<-v_j<b$,
from which it follows that $M' \in \ONCSM(n, v_j)$.

Assume $v$ is maximal with respect to $<_\cQ$.
Using Theorem~\ref{ndes-thm}, 
it is not hard to
show that
\[
v = v_1 c_1c_2\cdots c_k a_1 \bar{b_1} a_2 \bar{b_2} \cdots a_l \bar{b_l}
\]
for some numbers 
where $k,l \in \NN$ and
$c_1<c_2<\dots<c_k < a_1<a_2<\dots < a_l$ and $a_i < -b_i$ for $i \in [l]$.
If $v_1<0$
then 
$vt_0 \to v$ is the unique edge incident to $v$ in $\QQQ$,
and it follows from Theorem~\ref{ndes-thm}
that either $k = 0$ or $v_1<c_1$,
and in turn that $v=\alpha_{\max}(M)$ and $vt_0 = \alphaq_{\max}(M')$.
Instead suppose $v_1 > 0$.
Theorem~\ref{ndes-thm} then
implies that $k>0$, $v_1<-c_1$, $\{v_1,-c_1\} \in M$,
and $\{ \pm c_i\} \in M$ for $i \in [2,k]$.
In this case $vs_{12} \to v$ is the unique edge incident to $v$ in $\QQQ$.
Since $M$ is noncrossing, we must have $c_1 < -v_1 < c_i$ for all $i \in [2,k]$,
so $vs_{12}=\alphaq_{\max}(M')$.
\end{proof}

The following theorem confirms that $\QQQ$
is indeed a bipartite graph on the vertex set $\cA_n\sqcup \cQ_n$.

\begin{theorem}\label{qqq-thm}
The edges in 
$\QQQ$ have the following properties:
\ben
\item[(a)] If $v \in \cA_n\sqcup \cQ_n^+$ and $u\to v$ is an edge in $\QQQ$, then $u \in \cQ_n^-$ and $u\lessdot v$.
\item[(b)] If $v \in \cQ_n^+$ and $v\to w$ is an edge in $\QQQ$, then $w \in \cQ_n^-$
and $w\lessdot v$.
\een
\end{theorem}

\begin{proof}
Let $v \in \cA_n$ and suppose $u \to v$ is an edge in $\QQQ$.
We deduce that $u\lessdot v$ either
by Lemma~\ref{des-lem}(a) when $v_1<0$ or by Lemma~\ref{des1j-lem} when $v_1>0$.
Let $M = \sh(v)$ and define $M'$ from $M$ as in Lemma~\ref{chop-lem}.
It follows from Lemmas~\ref{prec-lem} and \ref{chop-lem}
that $u \leq_\cQ \alphaq_{\max}(M') \in \cQ_n^-$ so $u \in \cQ_n^-$.

Now let $v \in \cQ_n^+$ and suppose
 $u\to v$ and $v\to w$ are the edges incident to $v$ in $\QQQ$.
Then $u=vs_{1j}$ and $w = vt_{1i}$ where $i<j$ are the distinguished indices in $v$,
so  Corollary~\ref{dist-lem} implies that $u\lessdot v$ and $w\lessdot v$. 
Let 
$M = \shq(v)$.
Define $a = v_i$, $b = v_1$, and $c=-v_j$ so that $0<a<b<c$ and $\{a,c\} \in M$.
Construct $P$ (respectively, $Q$) from $M$ by replacing the blocks $\{a,c\}$ and $\{-a,-c\}$
by $\{a,b\}$ and $\{-a,-b\}$
(respectively, $\{b,c\}$ and $\{-b,-c\}$).
Since $M \in \ENCSM(n,b)$ it follows that $P \in \ONCSM(n,c)$ and $Q \in \ONCSM(n,a)$.
If $v$ is maximal with respect to $<_\cQ$ then
$v = \alphaq_{\max}(M)$
and evidently
$u = \alphaq_{\max}(P)$.
If $v$ is minimal with respect to $<_\cQ$ then 
$v = \alphaq_{\min}(M)$
and
$w = \alphaq_{\min}(Q)$.
Lemma~\ref{prec-lem} therefore
implies that $u \leq_\cQ \alphaq_{\max}(P) \in \cQ_n^-$
and $w \geq_\cQ \alphaq_{\min}(Q) \in \cQ_n^-$
so $u,w \in \cQ_n^-$.
\end{proof}

Let
$
\cS(w) 
= \cS_1(w)
$
and
$
\cT(w) 
= \cT^+_1(w)
$
for $w \in C_n$, where $\cS_j(w)$ and $\cT^\pm_j(w)$ are as in \eqref{stt-eq}.
If $b=w_1>0$
then we can also write
$
\cS(w) = \{ t_{ab} w:  a\in [b-1] \text{ and }w \lessdot t_{ab} w\} \sqcup \{w t_0\}
$
and
$
\cT(w) = \{ t_{bc} w:  c \in [b+1,n+1] \text{ and }w \lessdot t_{bc} w\}.
$

\begin{lemma}\label{st-lem}
Let $w,w' \in \cQ_n^-$ and $i \in [2,n-2]$. Suppose $b=w_1$ and $1 \leq a<b<c \leq n$.
\ben
\item[(a)] If $w \vartriangleleft_i w'$ and $w \lessdot t_{ab}w \in \cS(w)$, then $t_{ab} w \vartriangleleft_i t_{ab} w'$ and $w'\lessdot t_{ab}w' \in \cS(w')$.
\item[(b)] If $w'\vartriangleleft_i w$ and $w \lessdot t_{bc}w \in \cT(w)$ then $t_{bc}w' \vartriangleleft_i t_{bc} w$ and $w' \lessdot t_{bc} w' \in \cT(w')$. 
\een
\end{lemma}

\begin{proof}
Suppose $w \vartriangleleft_i w'$ and $w \lessdot t_{ab}w$.
Then for some $j \in [2,n]$ we have $w_j = -a$
and no $e \in \{w_2,w_3,\dots,w_{j-1}\}$ has $a<-e<b$.
We can write $w_iw_{i+1}w_{i+2} = zxy$ where $x<y<z$ and $0<z<-x$.
The only way we can fail to have 
$t_{ab} w \vartriangleleft_i t_{ab} w'$
is if $-b < w_{i+1} < -a = w_{i+2}$, but this would contradict $w \lessdot t_{ab}w$.
Since the relation $\vartriangleleft_i$ is length-preserving, we also have $w'\lessdot t_{ab}w'$.

Suppose $w' \vartriangleleft_i w$ and $w \lessdot t_{bc}w$.
Then for some $j \in [2,n]$ we have $w_j = c$
and no $e \in \{w_2,w_3,\dots,w_{j-1}\}$ has $b<e<c$.
We can write $w_iw_{i+1}w_{i+2} = yzx$ where $x<y<z$ and $0<z<-x$.
The only way we can fail to have 
$t_{bc} w' \vartriangleleft_i t_{bc} w$
is if $w_{i+1} =c > w_i > b$, but this would contradict $w \lessdot t_{bc}w$.
Since the relation $\vartriangleleft_i$ is length-preserving, we have $w'\lessdot t_{bc}w'$.
\end{proof}

\begin{lemma}\label{pre-st-lem1}
Let $u \in \cQ_n^-$ and $P = \shq(u)$.
Define $M$ by adding the block $\{\pm u_1\}$ to $P$.
Then $M \in \NCSM(n)$ and $ ut_0 \in \cA_M$, and $u \to ut_0$ is an edge in $\QQQ$.
\end{lemma}

\begin{proof}
Since $P$ has no blocks $\{a,b\}$ with $a<u_1<b$,
the matching $M $ belongs to $ \NCSM(n)$
and $c=u_1$ is the largest value such that $\{\pm c\} \in M$.
If $u  =\alphaq_{\max}(P)$ then evidently $ut_0 = \alpha_{\max}(M)$.
In general 
we have $u \leq_\cQ \alphaq_{\max}(P)$, and this implies that
$ut_0 \leq_\cQ \alpha_{\max}(M)$, so $ut_0 \in \cA_M$.
\end{proof}

\begin{lemma}\label{pre-st-lem2}
Let $u \in \cQ_n^-$, $P = \shq(u)$,
and $b=u_1$.
Suppose $a \in [b-1]$ is such that $\{\pm a\} \in P$ and $u \lessdot t_{ab}u$.
Define $M$ by removing the block $\{\pm a\}$ from $P$ and then adding $\{a,b\}$ and $\{-a,-b\}$.
Then $M \in \NCSM(n)$ and $t_{ab}u \in \cA_M$, and $u \to t_{ab}u$ is an edge in $\QQQ$.
\end{lemma}

\begin{proof}
Since $u\lessdot t_{ab}u$,
we must have $u_j=-a$ for some $j \in [2,n]$, and no numbers between
$- b$ and $- a$ can appear in $u_2u_3\cdots u_{j-1}$.
 There can be no blocks $\{x,y\} \in P$ with $x < a < y$: since $P \in \ONCSM(n,b)$, such a block necessarily satisfies $x = -y < a < y < b$, and then Lemma~\ref{qinterval-lem}(c)(iii) contradicts the previous sentence.
This is enough to conclude that $M \in\NCSM(n)$,
and that $e=a$ is the largest number with $\{ \pm e\} \in P$.
Assume $u=\alphaq_{\max}(P)$ is maximal under $<_\cQ$.
We must then have $j=2$ 
and we can write
\[u = b \bar a c_1c_2\cdots c_k a_1\bar {b_1} a_2\bar{b_2}\cdots a_l \bar{b_l}\]
where $\{ \pm c_i\}$ for $i \in [k]$ together with $\{\pm a\}$ are the symmetric blocks in $P$,
where $\{a_i,b_i\}$ for $i \in [l]$ are the blocks in $P$ with $0<a_i<b_i$,
and where $ -a < c_1  < \dots < c_k < 0 < a_1  < \dots <a_l$.
Hence 
\[
t_{ab}u = a \bar b c_1c_2\cdots c_k a_1\bar {b_1} a_2\bar{b_2}\cdots a_l \bar{b_l}
\]
and it is easy to see that
\begin{equation*}
    t_{ab}u \leq_\cA \alpha_{\max}(M) = c_1 c_2 \cdots c_k a_1 \bar{b_1} \cdots a \bar{b} \cdots a_l \bar{b_l}
\end{equation*}
 since
$ -b < c_i < 0 $ for each $i \in [k]$
and  $ -b < -b_i < 0 $ for each $i \in [l]$ with $a_i < a$ as $P$ is noncrossing.
Therefore $t_{ab}u \in \cA_M$.
If $u$ is not maximal under $<_\cQ$,
then it follows from Lemma~\ref{st-lem}(a)
that we still have $t_{ab}u \leq_\cA \alpha_{\max}(M)$ so again $t_{ab} u \in \cA_M$.
Once we know that $t_{ab}u \in \cA_M$, the claim that
$u \to t_{ab}u$ is an edge in $\QQQ$ holds by definition.
\end{proof}

\begin{theorem}\label{st-thm}
Let $u,w \in \cQ_n^-$ and $v \in C_n$ with $w_1 < n$.
\ben
\item[(a)]  It holds that $v \in \cS(u)$ if and only if $u\to v$ is an edge in $\QQQ$.

\item[(b)] It holds that $v \in \cT(w)$ if and only if $v \to w$ is an edge in $\QQQ$.
\een
\end{theorem}

\begin{proof}
Theorem~\ref{qqq-thm}
shows that if $u\to v$ is an edge in $\QQQ$ then $v \in \cS(u)$
and that if $v \to w$ is an edge in $\QQQ$ then $v \in \cT(w)$.
It remains to show the converse.

Let $P = \shq(u)$
and $v \in \cS(u)$.
If $v = ut_0$ 
or if 
$v = t_{ab}u$ where $0<a<b=u_1$ and $\{\pm a\} \in P$,
then $u \to v$ is an edge in $\QQQ$ by Lemmas~\ref{pre-st-lem1} and \ref{pre-st-lem2}.
Assume we are not in these cases.
There must be
 numbers $0<a<b<c=u_1$ with $\{a,b\} \in P$ and $v=t_{bc}u$.
Let $j \in [n]$ be such that $v_j = -b$.
By definition $P \in \ONCSM(n,c)$ has no blocks $\{x,y\}$ with $x<c<y$.
Since $\{u_2,u_3,\dots,u_{j-1}\}$ contains no numbers between $-c$ and $-b$ as $u\lessdot v$,
it follows from Lemma~\ref{qinterval-lem}(c) that $P$  has no blocks $\{x,y\}$ with $x<a<b<y$.
Form $M$ from $P$ by replacing the blocks $\{a,b\}$ and $\{-a,-b\}$ by $\{a,c\}$ and $\{-a,-c\}$.
Then $M \in \ENCSM(n,b)$,
and to show that $u\to v$ is an edge in $\QQQ$ it suffices to check that $v \in \cQ_M$.
If $u$ is maximal with respect to $<_\cQ$  then $u = \alphaq_{\max}(P)$
and
evidently $v  = \alphaq_{\max}(M)$. 
In general, Lemma~\ref{st-lem}(a) implies that $v \leq_\cQ \alphaq_{\max}(M)$
so $v \in\cQ_M$ as desired.

Next let  $Q = \shq(w)$ and $v \in \cT(w)$. 
Write $a=w_1$.
Since $a<n$ and $Q \in \ONCSM(n,a)$, we must have 
$v =  wt_{1i} = t_{ab}w$ for some $i \in [2,n]$ where $0<a<b=w_i$,
and there must exist a block $\{b,c\} \in Q$ with $b<c$.
By definition $Q$ has no blocks $\{x,y\}$ with $x<a<y$.
Since $\{w_2,w_3,\dots,w_{i-1}\}$ contains no numbers between $a$ and $b$ as $w\lessdot v$,
it follows from Lemma~\ref{qinterval-lem}(c) that $Q$ has no blocks $\{x,y\}$ with $x<b<c<y$.
Form $M$ from $Q$
by replacing $\{b,c\}$ and $\{-b,-c\}$ by $\{a,c\}$ and $\{-a,-c\}$. 
Then $M \in \ENCSM(n,b)$,
and
to show that $v\to w$ is an edge in $\QQQ$ it suffices to check that $v \in \cQ_M$.
If $w$ is minimal with respect to $<_\cQ$, then 
$w = \alphaq_{\min}(Q)$ and evidently $v = \alphaq_{\min}(M)$.
In general, it follows from Lemma~\ref{st-lem}(b) that $v \geq_\cQ \alphaq_{\min}(M)$
so $v \in \cQ_M$ as desired.
\end{proof}

The previous theorem does not apply when
$ w \in \cQ_n^-$ has $w_1 = n$,
since  then $\cT(w)$ consists of
the single element $(n+1)w_2w_3\cdots w_n n \in C_{n+1}$
but there are no edges $v \to w$ in $\QQQ$.

\begin{corollary}\label{source-cor}
A vertex $w \in \cA_n \sqcup \cQ_n$ is a source in $\QQQ$
if and only if $w_1 = n$,
in which case  $w \in \cQ_n^-$.
Thus the sources in $\QQQ$ are the elements $nv_1v_2\cdots v_{n-1} \in C_n$
where $v_1v_2\cdots v_{n-1} \in \cA_{n-1}$.
\end{corollary}

\begin{proof}
By definition 
no element in $ \cA_n \sqcup \cQ_n^+$ is a source in $\QQQ$.
Theorem~\ref{ndes-thm} implies that no atom $v\in \cA_n$ has $v_1=n$,
and that if $v \in \cQ_n^+$ then $v_i =-n$ for some $i \in [n]$.
It follows that an odd quasi-atom $w \in \cQ_n^-$
with $w_1=n$ cannot be the target of an edge $v\to w$ in $\QQQ$.

Suppose $w \in \cQ_n^-$ has $w_1 \in [n-1]$.
It remains to show that $w$ is not a source in $\QQQ$.
By Theorem~\ref{st-thm}, it suffices to check that $\cT(w)\neq\varnothing$.
Since $\shq(w)$ has no blocks $\{a,b\}$ with $a<w_1<b$,
the interval $[w_1+1,n]$ must be a non-empty union of blocks in $\shq(w)$.
It follows that $0<w_1<w_i$ for some $i \in [2,n]$,
and if $i$ is minimal with these properties then $wt_{1i} \in \cT(w)$,
so $\cT(w)\neq\varnothing$ as desired.
\end{proof}

\begin{corollary}\label{sink-cor}
A vertex $w \in \cA_n \sqcup \cQ_n$ is a sink in $\QQQ$ if and only if $w \in \cA_n$.
\end{corollary}

\begin{proof}
Since $w \to wt_0$ is always an edge in $\QQQ$ if $w \in \cQ_n^-$, this follows from the definition of $\QQQ$.
\end{proof}

For integers $0<m<n$, write ${\uparrow_m^n} : C_m \to C_n$ for
the transformation
\[
\label{shiftup-eq}
{\uparrow_m^n} (v_1v_2\cdots v_m) \omdef= n\cdots(m+3)(m+2)v_1v_2\cdots v_m(m+1) \in C_n.
\]
Recall that,
by convention, $C_0 $ is the set consisting of just the empty word $\emptyset$.
We define $\uparrow_0^n$ to be the map $\emptyset \mapsto n\cdots321$
and 
view $\ora{\cL_0}$ as the graph with no edges and a single vertex $\emptyset \in \cA_0$.

First define $\QQq{m,n}$ for $0 \leq m < n$ to be the directed graph given
by replacing
each vertex in $\ora{\cL_m}$
by its image under $\uparrow_{m}^{n}$.
One may interpret $\uparrow_{n}^{n+1}$
as the identity map $C_n \to C_n\hookrightarrow C_{n+1}$ and identify 
 $\QQq{n,n+1}$ with $\QQQ$.
%
Next define $\GGG$
to be the graph given by the disjoint union
\[\QQq{0,n + 1} \sqcup \QQq{1, n+1} \sqcup \QQq{2,n+1} \dots \sqcup \QQq{n, n+1}\]
with these additional edges:
for each $m \in [n]$ and $w \in \cA_{m-1}$,
include an edge from the sink 
\be
\label{sink}
 {\uparrow_{m-1}^{n+1}}(w_1w_2\cdots w_{m-1}) = (n+1)\cdots(m+2)(m+1) w_1w_2\cdots w_{m-1 }m
\ee
in $\QQq{m-1,n+1}$ to the source
\be
\label{source}
 {\uparrow_{m}^{n+1}}(m w_1w_2\cdots w_{m-1}) = (n+1)\cdots(m+3)(m+2) mw_1w_2\cdots w_{m-1}(m +1)
\ee
in $\QQq{m,n+1}$. 
Figure~\ref{fig3} shows $\GGG$ for $n=4$.

A vertex in $\QQq{m,n+1}$ is
\emph{odd}
if
it is the image under $\uparrow_{m}^{n+1}$ of an odd quasi-atom in $\ora{\cL_m}$.
All other vertices in $\QQq{m,n+1}$ or $\GGG$ are \emph{even};
in particular, the unique vertex ${\uparrow_{0}^{n+1}}(\emptyset)$ in $\QQq{0,n+1}$ is even.
Since every source in $\ora{\cL_m}$ is an odd quasi-atom and every sink is an atom,
the resulting division into even and odd vertices affords a bipartition of $\GGG$.

\begin{figure}[h]
\begin{center}
\begin{tikzpicture}[xscale=0.9,yscale=1.5]
\node (0) at (7.25,9) {$5\hs4\hs3\hs2\hs1$};
\node (1b) at (7.25,8) {$5\hs4\hs3\hs1\hs2$};
\node (1a) at (7.25,7) {$5\hs4\hs3\hs\bar1\hs2$};
\draw [->] (1b) -- (1a);
\node (2a) at (7.25,6) {$5\hs4\hs 2\hs \bar1\hs3$};
\node (2b1) at (2.5,5) {$5\hs4\hs 1\hs \bar2\hs3$};
\node (2b2) at (12,5) {$5\hs4\hs \bar2\hs \bar1\hs3$};
\draw [->] (2a) -- (2b1);
\draw [->] (2a) -- (2b2);
\node (3a2) at (2.5,4) {$5\hs3\hs 1\hs \bar2\hs4$};
\node (3b3) at (0,3) {$5\hs\bar3\hs 1\hs \bar2\hs4$};
\node (3b4) at (5, 3){$5\hs2\hs 1\hs \bar3\hs4$};
\node (3c) at (5,2) {$5\hs1\hs 2\hs \bar3\hs4$};
\node (3d) at (5,1) {$5\hs\bar1\hs 2\hs \bar3\hs4$};
\node (3a1) at (12,4) {$5\hs3\hs \bar2\hs \bar1\hs4$};
\node (3b1) at (10,3) {$5\hs\bar3\hs \bar2\hs \bar1\hs4$};
\node (3b2) at (14,3) {$5\hs2\hs \bar3\hs \bar1\hs4$};
\draw [->] (3a1) -- (3b1);
\draw [->] (3a1) -- (3b2);
\draw [->] (3a2) -- (3b3);
\draw [->] (3a2) -- (3b4);
\draw [->] (3b4) -- (3c);
\draw [->] (3c) -- (3d);
\node (4a2) at (0,0) {$4\hs \bar3\hs 1\hs \bar2$};
\node (4b3) at (-1,-1) {$\bar4\hs \bar3\hs 1\hs \bar2$};
\node (4b4) at (1,-1) {$3\hs \bar4\hs 1\hs \bar2$};
\node (4a4) at (5,0) {$4\hs \bar1\hs 2\hs \bar3$};
\node (4b7) at (3,-1) {$\bar4\hs \bar1\hs 2\hs \bar3$};
\node (4b8) at (7,-1) {$1\hs \bar4\hs 2\hs \bar3$};
\node (4b9) at (5,-1) {$3\hs \bar1\hs 2\hs \bar4$};
\node (4c1) at (5,-2) {$2\hs \bar1\hs 3\hs \bar4$};
\node (4d1) at (4,-3) {$\bar2\hs \bar1\hs 3\hs \bar4$};
\node (4d2) at (6,-3) {$1\hs \bar2\hs 3\hs \bar4$};
\node (4a1) at (10,0) {$4\hs \bar3\hs \bar2\hs \bar1$};
\node (4b1) at (9,-1) {$\bar4\hs \bar3\hs \bar2\hs \bar1$};
\node (4b2) at (11,-1) {$3\hs \bar4\hs \bar2\hs \bar1$};
\node (4a3) at (14,0) {$4\hs 2\hs \bar3\hs \bar1$};
\node (4b5) at (15,-1) {$\bar4\hs 2\hs \bar3\hs \bar1$};
\node (4b6) at (13,-1) {$3\hs 2\hs \bar4\hs \bar1$};
\node (4c2) at (13,-2) {$2\hs 3\hs \bar4\hs \bar1$};
\node (4d3) at (12,-3) {$\bar2\hs 3\hs \bar4\hs \bar1$};
\node (4d4) at (14,-3) {$1\hs 3\hs \bar4\hs \bar2$};
\draw [->] (4a1) -- (4b1);
\draw [->] (4a1) -- (4b2);
\draw [->] (4a2) -- (4b3);
\draw [->] (4a2) -- (4b4);
\draw [->] (4a3) -- (4b5);
\draw [->] (4a3) -- (4b6);
\draw [->] (4a4) -- (4b7);
\draw [->] (4a4) -- (4b8);
\draw [->] (4a4) -- (4b9);
\draw [->] (4b6) -- (4c2);
\draw [->] (4c2) -- (4d3);
\draw [->] (4c2) -- (4d4);
\draw [->] (4b9) -- (4c1);
\draw [->] (4c1) -- (4d1);
\draw [->] (4c1) -- (4d2);
\draw [->,dashed] (0) -- (1b);
\draw [->,dashed] (1a) -- (2a);
\draw [->,dashed] (2b2) -- (3a1);
\draw [->,dashed] (2b1) -- (3a2);
\draw [->,dashed] (3b1) -- (4a1);
\draw [->,dashed] (3b2) -- (4a3);
\draw [->,dashed] (3b3) -- (4a2);
\draw [->,dashed] (3d) -- (4a4);
\end{tikzpicture}
\end{center}
\caption{The directed graph $\ora{\cG_4}$.
The dashed arrows correspond to edges between vertices of the form
\eqref{sink} and \eqref{source}.
We have omitted the terminal 5 from all vertices in the final layer $\QQq{4,5}\subset \ora{\cG_4}$.
In contrast to what we see in this example, the graph $\GGG$ is not a directed tree for $n\geq 5$.}
\label{fig3}
 \end{figure}
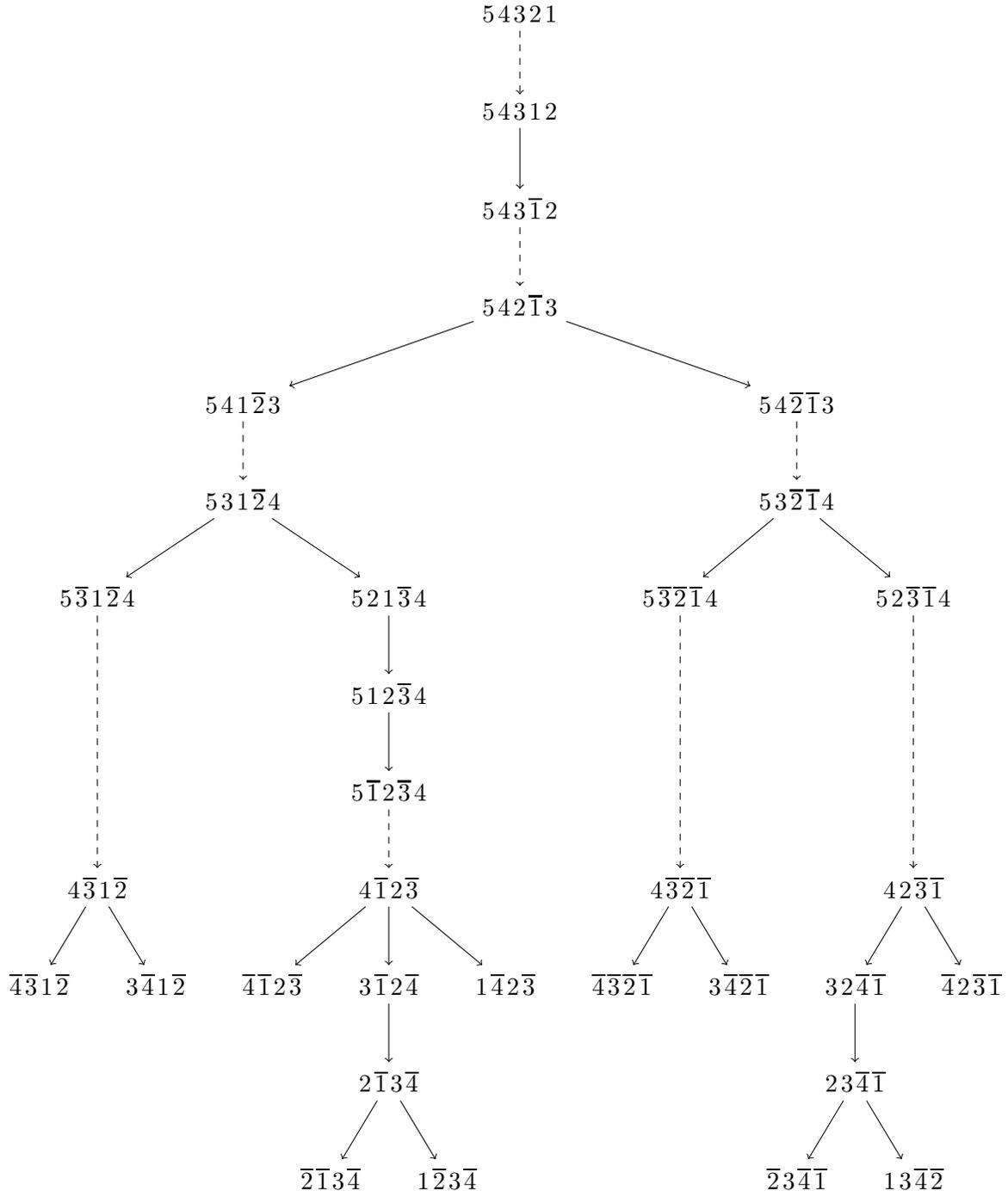

Recall that $w^C_n = \wB$ and $w^A_n = (n+1)n\cdots 321$ and $\delta_n = (n,n-1,\dots,3,2,1)$.

\begin{theorem*}[Theorem~\ref{thm:main-stanley}] \label{thm:inv-stanley-equality}
It holds that
$\iG_{w^C_n} = \G_{w^A_n} = S_{\delta_n}.$
\end{theorem*}

\begin{proof}
Let $w \in C_n$.
For $j \in [n]$, define $\cT_j^\pm(w)$ and $\cS_j(w)$ 
as in \eqref{st-eq}, and recall that $\cT(w) = \cT^+_1(w)$ and $\cS(w) = \cS_1(w)$.
If $0\leq m \leq n$ and $r=n-m+1$ and $\tilde w ={\uparrow_{m}^{n+1}}(w) \in C_n$, then
evidently 
\[\cT^+_{r}(\tilde w) = {\uparrow_{m}^{n+1}}(\cT(w)),
\qquad \cS_{r} (\tilde w)= {\uparrow_{m}^{n+1}}(\cS(w)),
\qquand
\cT^-_{r}(\tilde w) = \varnothing.
\]
Likewise,
if $u$ and $v$ are the elements \eqref{sink} and  \eqref{source}
then $\cT_{n-m}^+(v) = \{u\}$.
Theorem~\ref{st-thm}
implies
if $v$ is any odd vertex in $\GGG$,
then
$ \sum_{\{u \to v\} \in \GGG} \G_u = \sum_{\{v\to w\} \in \GGG} \G_w$.
Every even vertex in $\GGG$ has indegree at most 1 and outdegree at most 1.
Applying Lemma~\ref{lem:transition-graphs} for the function $f$
with $f(u) = \G_u$ if $u$ is an even vertex and $f(u) = 0$ if $u$ is an odd vertex
therefore gives
$\sum_{u \in \Source(\GGG)} \G_u = \sum_{u \in \Sink(\GGG)} \G_v.$
Since ${\uparrow_{0}^{n+1}}(\emptyset) = (n+1)n\cdots 321 = w^A_n$
is the unique source in $\GGG$ 
and 
since the set of sinks in $\GGG$
is precisely $ {\uparrow_{n}^{n+1}}(\cA_n) = \cA_n$,
we have $\iG_{w^C_n} = \G_{w^A_n}$.
The latter is $S_{\delta_n}$ by Theorems~\ref{intro-stan-thm} and \ref{super-thm}.
\end{proof}

\begin{corollary*}[Theorem~\ref{thm:main-enum}]
It holds that $|\iR(w^C_n)| = |\SYT(\delta_n)| = |\cR(w^A_n)|$.
\end{corollary*}

\begin{proof}
Let $N = \binom{n+1}{2} = \deg (S_{\delta_n}) = \ellhat(w^C_n)$.
Then $2^N |\iR(w^C_n)| = [x_1x_2\cdots] \iG_{w^C_n} =  [x_1x_2\cdots] S_{\delta_n} $,
which is the number of marked standard tableaux of shape $\delta_n$.
Since this number is evidently $2^N |\SYT(\delta_n)|$,
and since we have already seen that $|\SYT(\delta_n)| = |\cR(w^A_n)|$, 
the result follows.
\end{proof}

\section{Future directions}







\subsection{Geometry}

There are geometric connections in type A for which we do not know type C analogues. The type A involution Stanley symmetric function $\iF_w$ is a limit of \emph{involution Schubert polynomials}, which are known to represent the cohomology classes of the orbit closures of $\operatorname{O}_n(\CC)$ acting on the type A complete flag variety. One can also define type C involution Schubert polynomials, which represent cohomology classes on the type C isotropic flag variety insofar as they are positive integer combinations of type C Schubert polynomials, but we do not know a more interesting description of these classes.

\subsection{Positivity} \label{positivity-sect}

As mentioned in Section~\ref{prelim-stan-sect}, $\iF_y$ is not only Schur-positive but Schur-$Q$-positive, with integral coefficients up to a predictable scalar \cite{HMP4}. It would be interesting to find a similar expression for  $\iG_y$ as a positive combination of some Schur-$Q$-positive symmetric functions in a nontrivial way. Theorem~\ref{thm:main-stanley} accomplishes this for $\iG_{w_n^C}$, because the Schur $S$-functions are Schur-$Q$-positive (they are in fact \emph{skew Schur $Q$-functions}), but usually $\iG_y$ is not Schur-$S$-positive.

\subsection{Type D analogues}

Let $D_n$ be the subgroup of signed permutations in $C_n$ whose one-line representations have an even number of negative letters.
This is a finite Coxeter group of classical type D relative to the generating set $S = \{t_1',t_1,t_2,\dots,t_{n-1}\}$ where $t_1' \omdef= t_0t_1t_0$.
For $w \in D_n$ and $a \in \Red(w)$, let $\flt{a}$ be the word obtained from $a$ by replacing each $t_1'$ with $t_1$,
and define $\flt{\Red}(w) = \{ \flt{a} : a \in \Red(w)\}$.
 For instance, $\Red(\bar 1 3 \bar 2) = \{(t_1, t_1', t_2), (t_1', t_1, t_2)\}$ while $\flt{\Red}(\bar 1 3 \bar 2) = \{(t_1, t_1, t_2)\}$.

In type D it is the sets $\flt{\Red}(w)$ that have simple tableau enumerations. Let $w_n^D$ be the longest element of $D_n$.
One has $w_n^D = w_n^C = \bar 1\hs \bar 2 \cdots \bar n$ if $n$ is even and $w_n^D = 1 \bar2 \hs \bar 3 \cdots \bar n$ if $n$ is odd.

\begin{theorem}[{Billey and Haiman \cite[Proposition 3.9]{BH}}] \label{d-thm}
If $n\geq 3$ then $|\flt{\Red}(w_n^D)| = |\SYT((n-1)^n)|$, which is also the number of (unmarked) shifted standard tableaux of shape $(2n-2, 2n-4, \ldots, 2)$.
\end{theorem}

Let $(W,S)$ be a Coxeter system with a group automorphism $\theta:W\to W$ such that $\theta(S) = S$ and $\theta=\theta^{-1}$.
The  set of \emph{twisted involutions} with respect to $\theta$ is $\I_\theta(W) = \{w \in W : \theta(w) = w^{-1}\}$. The set of \emph{(twisted) atoms} $\cA_\theta(y)$ of $y \in \I_\theta(W)$ consists of the minimal-length elements $w \in W$ with $\theta(w)^{-1} \circ w = y$, and the set of \emph{(twisted) involution words} is $\iR_\theta(y) = \bigsqcup_{w \in \cA_\theta(y)} \Red(w)$.

Assume $W$ is finite with longest element $w_0$.
If $W$ is $A_n$, $C_n$, or $D_{2n+1}$ for $n > 1$, then the only possibilities for $\theta$ are the identity map and 
$w \mapsto  w_0 w w_0$, and 
it holds that $|\iR_\theta(w_0)| = |\iR(w_0)|$ and (in type D) $|\hs\flt{\iR}_\theta(w_0)| = |\hs\flt{\iR}(w_0)|$ by \cite[Corollary 3.9]{HMP2}. Define $*$ as the automorphism of $D_n$ which interchanges $t_1$ and $t_1'$ and fixes $t_i$ for $i \in [2,n-1]$.
When $n$ is odd, $*$ is the inner automorphism $w \mapsto w_0 ww_0$.
There appear to be involution word analogues of Theorem~\ref{d-thm}:

\begin{conjecture} If $n \geq 3$ then $\left|\hs\flt{\iR}(w_n^D)\right| =|\SYT(\lambda)|$ and $\left|\hs\flt{\iR}_*(w_n^D)\right| = |\SYT(\mu)|$ for 
the partitions
$\lambda = (n-1, n-2, \ldots, \lfloor \frac{n}{2} \rfloor, \lfloor \frac{n}{2} \rfloor, \ldots, 2, 1)$
and 
$\mu = (n-1, n-2, \ldots, \lceil \frac{n}{2} \rceil-1, \lceil \frac{n}{2} \rceil-1, \ldots, 2, 1).$
\end{conjecture}

For $n=3,4,5,6$, we have checked by computer that $\left|\hs\flt{\iR}(w_n^D)\right|=3$, $70$, $5775$, $10720710$
and $\left|\hs\flt{\iR}_*(w_n^D)\right| = 3$, $35$, $5775$, $3573570$ as predicted by this conjecture.

\end{document}